\documentclass[oneside, a4paper,11pt,reqno]{amsart}
\usepackage{color}

\usepackage{mathrsfs,amssymb,amsmath,amsthm}
\usepackage[T1]{fontenc}

\newtheorem{theorem}{Theorem}[section]

\newtheorem{hypo}{Hypothesis}

\newtheorem{prop}[hypo]{Proposition}

\newtheorem{lem}[hypo]{Lemma}

\def\A{\mathcal{A}}

\def\PP{\mathbb{P}}
\def\RR{\mathbb{R}}
\def\ZZ{\mathbb{Z}}

\def\EE{\mathbb{E}}
\def\NN{\mathbb{N}}

\newcommand {\sur}[2] { \stackrel {\scriptstyle{#1}}{#2}}

\begin{document}

 \vglue50pt

\centerline{ \Large \bf  The quenched limiting distributions}
\centerline{ \Large \bf  of a charged-polymer model}

\bigskip

\medskip
 \renewcommand{\thefootnote}{\fnsymbol{footnote}}

 \centerline{Nadine Guillotin-Plantard\footnotemark[1], Renato Soares dos Santos\footnotemark[1]}

\footnotetext[1]{Institut Camille Jordan, CNRS UMR 5208, Universit\'e de Lyon, Universit\'e Lyon 1, 43, Boulevard du 11 novembre 1918, 69622 Villeurbanne, France.\\ E-mail: nadine.guillotin@univ-lyon1.fr ; soares@math.univ-lyon1.fr;\\ This work was supported by the french ANR project MEMEMO2 10--BLAN--0125--03.}
\medskip

 \centerline{\it Universit\'e Lyon 1}

\bigskip
{\leftskip=2truecm
\rightskip=2truecm
\baselineskip=15pt
\small

\noindent{\slshape\bfseries Summary.}       
The limit distributions of the charged-polymer Hamiltonian of Kantor and Kardar [Bernoulli case] and Derrida, Griffiths and Higgs [Gaussian case] are considered. 
Two sources of randomness enter in the definition: a random field $q= (q_i)_{i\geq 1}$ of i.i.d.\ random variables, which is called the random \emph{charges}, and a random walk $S = (S_n)_{n \in \NN}$ evolving in $\ZZ^d$, independent of the charges. The energy or Hamiltonian  $K = (K_n)_{n \geq 2}$ is then defined as 
$$K_n := \sum_{1\leq i<j\leq n} q_i q_j {\bf 1}_{\{S_i=S_j\}}.$$
The law of $K$ under the joint law of $q$ and $S$ is called ``annealed'', and the conditional law given $q$ is called ``quenched''. 
Recently, strong approximations under the annealed law were proved for $K$. In this paper we consider the limit distributions of $K$ under the quenched law.

\medskip



} 

\medskip
 
\section{Introduction}
Let $d\geq 1$ and $q= (q_i)_{i \geq 1}$ be a collection of i.i.d.\ real random variables, hereafter referred to as {\it charges}, and $S = (S_n)_{n \ge 0}$ be a random walk in $\mathbb{Z}^d$ starting at $0$,
i.e., $S_0 = 0$ and
$\left(S_n-S_{n-1}\right)_{n \ge 1}$
is a sequence of i.i.d.\ $ \mathbb{Z}^d $-valued random variables, independent of $q$.
We are interested in the limit distributions of the sequence $K := (K_n)_{n \ge 1}$ defined 
by setting $K_1 := 0$ and, for $n \geq 2$,
\begin{equation}
K_n := \sum_{1\leq i<j\leq n}  q_{i} q_j {\bf 1}_{\{S_i=S_j\}}.
\end{equation}
In the physics literature this sum is known as the Hamiltonian of the so-called {\it charged polymer model }; see Kantor and Kardar \cite{KK} in
the case of Bernoulli random charges and Derrida, Griffiths and Higgs \cite{DGH} in the Gaussian case.
This model has been largely studied by physicists since it is believed that a protein molecule looks like a random walk with random charges attached at the vertices of the walk; these charges are interacting
through local interactions mimicking chemical reactions \cite{PM}. 

Results were first established under the {\it annealed} measure, that is when one averages at the same time over the charges and the random walk.
Chen \cite{Ch08}, Chen and Khoshnevisan \cite{CK} proved that the one-dimensional limiting distributions are closely related to the model of {\it Random walk in random scenery}. 
Hu and Khoshnevisan \cite{HK} then established that in dimension one the limit process of the (correctly renormalized) Hamiltonian $K_n$ is strongly approximated by a Brownian motion, time-changed
 by the self-intersection local time process of an independent  Brownian motion.
Especially, it differs from the so-called Kesten and Spitzer's process \cite{KS79} obtained as the continuous limit process of the one-dimensional random walk in random scenery.

To our knowledge, distributional limit theorems for {\it quenched} charges (that is, conditionally given the charges) are not known. Let us note that in the physicists' usual setting the charges
are usually quenched: a typical realization of the charges is fixed, and the average is over the
walk. In the case of dimension one, we determine  the  quenched  weak limits of $K_n$ by applying  Strassen's \cite{Strassen} functional law of the iterated logarithm. As a consequence, conditionally on the random charges, the Hamiltonian $K_n$ does not converge in law. In contrast with the one-dimensional setting, we show that the quenched central limit theorem holds for random walks in dimensions $d\ge2$ with finite non-singular covariance matrix and centered, reduced charges. In $d \ge 3$ we obtain convergence to Brownian motion under the standard scaling $\sqrt{n}$, while in $d=2$ we are only able to show convergence of the finite-dimensional distributions under the unusual $\sqrt{n\log n}$-scaling and an additional condition on the charges, namely a moment of order strictly larger than $2$. We also provide in the appendix a proof of a functional central limit theorem under the conditional law given the random walk. In particular, our results imply annealed functional central limit theorems under weaker assumptions on the charges than the ones in \cite{CK,HK}.

\section{Case of dimension one}

\subsection{Results}

In this section we study the case of the dimension one,  $S=(S_k, k\geq 0) $ is the simple one-dimensional random walk. 
Moreover we assume that
$$\EE(q_1)=0, \  \EE( q_1^2)=1 \ \text{and}\   \EE( |q_1|^{6}) <\infty.$$ 
We prove that under these assumptions, there is no quenched distributional limit theorem for $K$. In the sequel, for $0 < b \le \infty$, we will denote by $\mathcal{AC}([0,b]\to {\mathbb R})$ the set of absolutely continuous functions defined on the interval $[0,b]$ with values in $\mathbb{R}$. Recall that  if $f\in\mathcal{AC}([0,b]\to {\mathbb R}) $, then the derivative of $f$ (denoted by $\dot{f} $) exists  almost everywhere and is  Lebesgue integrable  on
 $[0,b]$. Define  \begin{equation} \mathcal{K}^*  := \Big\{ f\in \mathcal{AC}( \RR_+ \to \RR) :  f(0)=0, \int_{0}^{\infty} (\dot{f}(x))^2 dx  \le 1\Big\}. \label{K*}  \end{equation}

\begin{theorem}\label{T:0} 
For $\mathbb{P}$-a.e.\ $q$, under the quenched probability $\mathbb{P} \left(. \mid q \right)$, the process 
 $$\tilde{K}_n:= \frac{K_n} {(n^{3/2} \log \log n)^{1/2}},  \qquad n >e^e, $$ 
does not converge in law. 
More precisely, for $\mathbb{P}$-a.e.\ $q$, under the quenched probability $\mathbb{P} \left(. \mid q \right)$, 
the limit points of the  law of 
$\tilde{K}_n,$ as $n \to \infty,  $
 under the topology of  weak convergence of measures,  are  equal to the set of the laws of random variables  in $\Theta_B $,  with  
 \begin{equation} \Theta_B  :=  \Big\{  f(V_1) :    f \in  {\mathcal K}^*  \Big\},  \label{KB} \end{equation} 
 where     $V_1$ denotes the self-intersection local time at time $1$ of a one-dimensional Brownian motion $B$ starting from $0$.
 
The set $\Theta_B$ is closed for the topology of  weak convergence of measures, and is a compact subset of $L^2 ( (B_t)_{t\in[0,1]})$.
 
\end{theorem}
  
Instead of Theorem \ref{T:0}, we shall prove that there is no quenched limit theorem for the continuous analogue of $K$ introduced by Hu and Khoshnevisan \cite{HK} and deduce  Theorem \ref{T:0}  by using a strong approximation. Let us define this continuous analogue: Assume that  $B:=(B(t))_{t\geq 0}$, $W:=(W(t))_{t\geq 0}$ are  two real Brownian motions starting from $0$,  defined on the same probability space and independent of each other. We denote by $\mathbb{P}_{B}$, $\mathbb{P}_{W}$ the law of these processes.  We will also denote by $(L_t(x))_{t\geq 0,x\in\RR}$ a continuous version with compact support of the local time of the process $B$, and $(V_t)_{t\geq 0}$ its self-intersection local time up to time $t$, that is
$$V_t:=\int_{\RR} L_t(x)^2 \, dx.$$  
We define the continuous version of the sequence $K_n$ as
$$Z_t:= W( V_t ), t\geq 0.$$
In dimension one, under the annealed measure,  Hu and Khoshnevisan \cite{HK} proved that the process 
$(n^{-3/4} K([nt]) )_{t\geq 0}$  weakly converges in the space of continuous functions to the continuous process $Z= (2^{-1/2} Z_t)_{t\geq 0}$. 
They gave a stronger version of this result, more precisely, they proved that there is a coupling of $q$, $S$, $B$
 and $W$ such that $(q, W)$ is independent of $(S,B)$ and for any $\varepsilon \in (0,1/24)$, almost surely,
\begin{equation}\label{Zh}
K_n = 2^{-1/2} Z_n + o(n ^{\frac{3}{4} - \varepsilon} ) , \  \  n\rightarrow +\infty.
\end{equation}
Theorem \ref{T:0} will follow from this strong approximation and the following result.

\begin{theorem}\label{T:1} $\PP_W$-almost surely,  under the quenched probability $\PP(\cdot \vert W)$,  the limit points of the  law of 
$$\tilde{Z}_t:= \frac{Z_t} {(2 t^{3/2} \log \log t )^{1/2}},  \qquad t \to \infty,  $$
 under the topology of  weak convergence of measures,  are  equal to the set of the laws of random variables  in $\Theta_B $ defined in Theorem \ref{T:0}. Consequently, under $\PP(\cdot \vert W)$, as $t \to \infty$,   $\tilde{Z}_t$ does not converge in law. 
 
 \end{theorem}

To prove Theorem \ref{T:1}, we shall   apply  Strassen \cite{Strassen}'s functional law of the iterated logarithm applied to the Brownian motion $W$.
\subsection{Proofs}

For a one-dimensional Brownian motion  $(W(t), t \geq 0)$ starting from $0$, let us define for any $ \lambda > e^e$, 
$$ W_{\lambda} (t):= \frac{W(\lambda t)}{ (2 \lambda \log \log \lambda)^{1/2}},  \qquad   t \geq 0. $$ 

\begin{lem} \label{LIL} 

(i)   Almost surely, for any $r>0$ rational numbers, $(W_\lambda(t), 0\le t \le r)$ is  relatively compact in the uniform topology and the set of its limit points is  $\mathcal{K}_{0,r}$, with  $$\mathcal{K}_{0,r}:=\Big\{ f\in \mathcal{AC}([0,r] \to {\mathbb R} ) :  f(0)=0, \int_0^{r} (\dot{f}(x))^2 dx  \le 1\Big\}.$$

(ii) There exists some finite random variable $\A_W$ only depending on $(W(x), x \geq 0)$ such that for all $\lambda \ge e^{36}$,  
$$    \sup_{ t > 0} \frac{\vert W_\lambda(t)\vert}{\sqrt{ \vert t\vert  \log \log (  \vert t\vert + \frac{1}{\vert t\vert } + 36)}} =:\A_W  < \infty .$$
\end{lem}

The proof of this lemma can be found in \cite{GHS}.

Let us define for all $\lambda > e^e$ and $ n\ge 1$, 
$$ H_\lambda:= W_\lambda(V_1), \qquad H_\lambda^{(n)}:= W_\lambda(V_1)\,  {\bf 1}_{\{V_1\leq n\}},$$

\begin{lem}\label{L9}  There exist some positive constants $c_1,c_2$ such that for any $\lambda> e^{36}$ and $n \ge 1$, we have  
\begin{eqnarray}    \EE_B \Big\vert   H_\lambda- H_\lambda^{(n)} \Big\vert &\le&  c_1\,  e^{- c_2{n^2}}\, \A_W, \label{h2nb} 
    \\ \EE_B \Big \vert  f(V_1)  \Big\vert{\bf 1}_{\{V_1> n\}}  & \le&   c_1\,  e^{- c_2{n^2}}, \label{fL(x)n} 
\end{eqnarray}   for any function $f \in  \mathcal K^*$.
\end{lem}

{\noindent\bf Proof:}  
By Lemma \ref{LIL} (ii),  $\EE_B\big[ (W_\lambda(V_1))^2\big]  \le \A_W^2\, \EE_B \big[ \vert V_1 \vert \log \log ( \vert V_1\vert + \frac{1}{\vert V_1 \vert}+36)\big]
\leq c_1 \A_W^2$, since $V_1$ has finite moments of any order. Then by Cauchy-Schwarz' inequality, we have that 
\begin{eqnarray*}
\EE_B \Big\vert   H_\lambda- H_\lambda^{(n)} \Big\vert &=& \EE_B \Big[ W_\lambda(V_1)  1_{(V_1 >n)} \Big] \\
	 & \le& \sqrt{\EE_B \Big[ W_\lambda(V_1)^2\Big] } \, \sqrt{ \PP_B\Big( V_1 >n\Big)} \\
	 &\le& c_1 \, \A_W\,  e^{- {c_2 n^2}},
\end{eqnarray*}

\noindent by the fact that: $\PP_B\big( V_1  > x\big) \le c_1 e^{- c_2 x^2}$ for any $x>0$ (see Corollary 5.6 in \cite{KL98}). Then  we get \eqref{h2nb}. 

For the other part of the lemma, let $f \in \mathcal K^*$, observe that $\vert f(x) \vert \le \sqrt{\Big\vert x  \int_0^x (\dot{f}(y))^2 dy \Big\vert} \le \sqrt{\vert x\vert}$ for all $x \in \RR_+$.
 Then by Cauchy-Schwarz' inequality, we have that 
\begin{eqnarray*}
\EE_B \Big[\vert f(V_1) \vert 1_{(V_1 >n)} \Big] & \le& \sqrt{\EE_B \Big[ f(V_1)^2\Big] } \, \sqrt{ \PP_B\Big( V_1  >n\Big)} \\
	 & \le& \sqrt{\EE_B \Big[ V_1\Big] } \, \sqrt{ \PP_B\Big( V_1  >n\Big)} \\
	 &\le& c_1 \,  e^{- {c_2 n^2}}.
\end{eqnarray*}
 Then  \eqref{fL(x)n} follows.  $\Box$

Let $L^{1}(B)$ be the set of real random variables which are $\sigma( B_t, t\geq 0)$-measurable and $\PP_{B}$-integrable.
For any $X \in L^{1}(B)$ and any subset $\Theta$ of $L^{1}(B)$, we denote by $d_{L^{1}(B)}(X,\Theta)$ the usual distance
$\inf_{Y\in \Theta} \EE_{B}[ | X - Y | ] $.

\begin{lem}  \label{L:conv} $\PP_W$-almost surely, $$d_{L^1(B)}( H_\lambda, \Theta_B) \to 0, \qquad \mbox{ as } \lambda \to \infty,$$ where  $\Theta_B$ is defined in \eqref{KB}.    Moreover, $\PP_W$-almost surely,
 for any $\zeta \in  \Theta_B$,  
$$ \liminf_{\lambda \to \infty} d_{L^1(B)}( H_\lambda, \zeta)=0.$$
\end{lem}

{\noindent\bf Proof:}  Let $\varepsilon>0$.  Choose a large $n=n(\varepsilon)$ such that $c_1 e^{- c_2 n^2} \le \varepsilon$, where $c_1, c_2$ are the constants defined in Lemma \ref{L9}. 
 By Lemma \ref{LIL} (i),  for all large $\lambda \ge \lambda_0(W, \varepsilon, n)$, there exists some function $g =g_{\lambda, W, \varepsilon, n}\in {\mathcal K}_{0, n}$ 
 such that $\sup_{ x \in [0, n]} \vert W_\lambda(x) - g(x) \vert \le \varepsilon$.   We get that $$ \EE_B \Big\vert H_\lambda^{(n)} - g(V_1)  {\bf 1}_{\{V_1\leq n\}} \Big\vert \le  \varepsilon.$$ 

We extend $g$ to $\RR_+$ by letting $g(x)= g(n)$ if $x\ge n$, then $g \in \mathcal K^*$.   By the triangular inequality, \eqref{h2nb} and \eqref{fL(x)n},  
$$ \EE_B \Big\vert H_\lambda  - g(V_1) \Big\vert  \le (2 + \A_W) \varepsilon. $$

It follows that $d_{L^1(B)}( H_\lambda, \Theta_B) \le (2 + \A_W) \varepsilon$. Hence $\PP_W$-a.s., $\limsup_{\lambda\to\infty} d_{L^1(B)}( H_\lambda, \Theta_B) \le (2 + \A_W) \varepsilon$, showing the first part in the lemma.

For the other part of the Lemma, let $h \in \mathcal K^*$ such that $\zeta=  h(V_1) $.  For any $\varepsilon>0$,  we may use \eqref{fL(x)n} and choose an integer $n=n(\varepsilon)$ such that $c_1  e^{-c_2 n^2} \le \varepsilon$ and $$ d_{L^1(B)}( \zeta, \zeta_n)   \le \varepsilon,$$

\noindent where $\zeta_n:=  h(V_1) {\bf 1}_{\{V_1\leq n\}}$.  Applying Lemma \ref{LIL} (i) to the restriction of $h$ on $[0, n]$, we may find  a sequence $\lambda_j= \lambda_j(\varepsilon, W, n) \to \infty$ such that $\sup_{ \vert x \vert \le n} \vert W_{\lambda_j}(x) - h(x) \vert \le \varepsilon$, then  
 $$ d_{L^1(B)}( H^{(n)}_{\lambda_j}, \, \zeta_n) \le \varepsilon.$$

By \eqref{h2nb} and the choice of $n$, $ d_{L^1(B)}( H^{(n)}_{\lambda_j}, H_{\lambda_j}) \le \varepsilon \A_W$ for all large $\lambda_j$, it follows from the triangular inequality that $$ d_{L^1(B)} ( \zeta, H_{\lambda_j}) \le (2 + \A_W )\varepsilon,$$
implying that $\PP_W$-a.s., $ \liminf_{\lambda \to \infty} d_{L^1(B)}( H_\lambda, \zeta) \le (2+ \A_W )\varepsilon \to 0$ as $\varepsilon\to0$.   $\Box$
  
 We now are ready to give the proof of Theorems \ref{T:1} and \ref{T:0}.
 
  {\noindent\bf Proof of Theorem \ref{T:1}.} Remark that $\PP_W$-a.s.,  \begin{equation}\label{Z}
W(V_t) \sur{(d)}=  W(V_1 t^{3/2})   \end{equation} 
from  the scaling property of the self-intersection local time of the Brownian motion $B$. 
The first part of Theorem \ref{T:1} directly follows from Lemma \ref{L:conv}. 
 
{\noindent\bf Proof of Theorem \ref{T:0}.}
We use the strong approximation of \cite{HK} : there exists on a suitably enlarged probability space,  a coupling of $q$, $S$, $B$
and $W$ such that $(q, W)$ is independent of $(S,B)$ and for any $\varepsilon \in (0,1/24)$, almost surely,
$$
K_n = 2^{-1/2} Z_n + o(n ^{\frac{3}{4} - \varepsilon} ) , \  \  n\rightarrow +\infty.
$$
From the independence of $(q, W)$ and $(S,B)$, we deduce that for $\PP$-a.e. $(q, W)$, under the quenched probability $\PP(.\vert q,W)$, the limit points of the 
 laws of $\tilde{K}_n$ and $\tilde{Z}_n$ are the same ones. 
Now, by adapting the proof of Theorem \ref{T:1}, we have that for $\PP$-a.e. $(q, W)$, under the quenched probability $\PP(.\vert q,W)$, the limit points of the 
 laws of  $\tilde{Z}_n$, as $n\rightarrow \infty$,  under the topology of  weak convergence of measures,  are  equal to the set of the laws of random variables  in 
 $\Theta_B $. It gives that for $\PP$-a.e. $(q, W)$, under the quenched probability $\PP(.\vert q,W)$, the limit points of the 
 laws of  $\tilde{K}_n$, as $n\rightarrow \infty$,  under the topology of  weak convergence of measures,  are  equal to the set of the laws of random variables  in 
 $\Theta_B $ and Theorem \ref{T:0} follows.
 
 Let $(\zeta_n)_n$ be a sequence of random variables in $\Theta_B$, each $\zeta_n$ being associated to a function $f_n\in \mathcal{K}^*$. The sequence of the (almost everywhere) derivatives of $f_n$ is then a bounded sequence in the Hilbert space $L^2(\RR_+)$, so we can extract a subsequence which weakly converges to a limit whose integral is in $\mathcal{K}^*$. Using the definition of the weak convergence and the fact that $\zeta_n=\int_{\RR_+} 1_{[0,V_1]}(y) \dot{f}_n(y)\, dy$, $(\zeta_n)_n$ converges almost surely. Since the sequence $(\zeta_n)_n$ is bounded in $L^p(B)$ for any $p\geq 1$, the convergence also holds in $L^2(B)$, and compactness follows. $\Box$

\section{Case of dimension two}
\label{sec:dim2}
\subsection{Assumptions and results}
\label{subsec:assumptions_results_d2}

We will make the following two assumptions on the random walk and on the random scenery:

\paragraph{(A1)}
The random walk increment $S_1$ takes its values in $\ZZ^d$ and has a centered law with a finite and non-singular covariance matrix $\Sigma$.
We further suppose that the random walk is aperiodic in the sense of Spitzer \cite{S76},
which amounts to requiring that $\varphi(u) = 1$ if and only if $u \in 2 \pi \ZZ^d$,
where $\varphi$ is the characteristic function of $S_1$.

\paragraph{(A2)} $\EE[q_1]=0$,  $\EE[q_1^2]=1$ and $\mathbb{E}[|q_1|^\gamma]<\infty $ for some $\gamma>2$.



Our aim is to prove the following quenched central limit theorem. 
\begin{theorem}\label{theoCLT}
Assume (A1), (A2) and $d = 2$.
Then, for any $0 < t_1 < \ldots < t_N < \infty$,
\begin{equation}\label{eq:theoCLT}
\left( \frac{K_{\lfloor n t_1 \rfloor}}{\sqrt{n \log n}}, \ldots, \frac{K_{\lfloor n t_N \rfloor}}{\sqrt{n \log n}} \right) \Rightarrow
(B_{t_1}, \ldots, B_{t_N}) \; \text{ under } \PP(\cdot | q) \text{ for } \PP\text{-a.e.\ } q,
\end{equation}
where ``$\Rightarrow$'' denotes convergence in distribution as $n \to \infty$, and
$B$ is a Brownian motion with variance $\sigma^2 = (2 \pi \sqrt{\det \Sigma})^{-1}$.
\end{theorem}

\noindent\textbf{Remark:}
The conclusion of this theorem still holds if, alternatively, the assumptions (A1) and $d=2$ are replaced 
by the following:
\paragraph{(A1')}   The sequence $S = (S_n)_{n \ge 0}$ is an aperiodic random walk in $\ZZ$ starting from $0$ 
such that the sequence $\left(\frac{S_n}n\right)_n$ converges
in distribution to a random variable with characteristic function given by 
$t\mapsto \exp(-a |t|)$ with $a>0$, in that case $\sigma^2$ is given by  $(2\pi a)^{-1}$.

Indeed, the proof of Theorem \ref{theoCLT} depends on $S$ through properties of the self-intersection local time and of the intersection local time of the random walk $S$ which are known to be the same under assumptions (A1) in $d=2$ or (A1') in $d=1$.

\noindent\textbf{Remark:}
Theorem~\ref{theoCLT} implies convergence of finite-dimensional distributions of $K_{\lfloor n t \rfloor} / \sqrt{n \log n}$ under the quenched law
in any countable set of times $t$. If additionally tightness in Skorohod space
can be established, this will imply functional convergence to Brownian motion.

An ingredient in the proof of Theorem \ref{theoCLT} is the following functional central limit theorem
under $\mathbb{P}(\cdot |S)$, which is of independent interest. 
Indeed, it implies the same result under the annealed law, improving the previously
known assumptions for such a theorem to hold (see \cite{HK}).

Let
\begin{equation}
s_n^2 := \left\{ \begin{array}{lcl}
n \log n & \text{ if } & d =2,\\
n & \text{ if } & d \ge 3,
\end{array}\right.
\end{equation}
and
\begin{equation}
\sigma^2 := \left\{\begin{array}{lcl}
(2 \pi \sqrt{\det \Sigma})^{-1} & \text{ if } & d = 2,\\
\sum_{n=1}^{\infty} \PP(S_n=0) & \text{ if } & d \ge 3.
\end{array}\right.
\end{equation}

\begin{theorem}\label{thm:FCLTgivenS}
Under conditions (A1)--(A2) and $d\ge2$, or (A1')--(A2), for a.e.\ realization of $S$, the process
\begin{equation}
B^{(n)}_t := s_n^{-1} K_{\lfloor n t\rfloor}, t \ge 0,
\end{equation}
converges weakly under $\PP(\cdot | S)$ in the Skorohod topology as $n \to \infty$ to a Brownian motion with variance $\sigma^2$.
\end{theorem}

The proof of Theorem~\ref{thm:FCLTgivenS} is an application of the martingale CLT, 
and is given in Appendix~\ref{sec:appendixFCLT}.

The proof of Theorem~\ref{theoCLT} will be given in two steps as follows.
Define the subsequence
\begin{equation}\label{deftaun}
\tau_n := \lceil \exp n^\alpha \rceil, \;\;  \frac{1}{2} \vee \frac{2}{\gamma} < \alpha < 1.
\end{equation}

Then the following two propositions directly imply Theorem~\ref{theoCLT}. Both assume $d=2$ and (A1)--(A2).
\begin{prop}\label{prop:CLTsubseq}
For any $0 < t_1 < \cdots < t_N < \infty$, 
\begin{equation}\label{eq:CLTsubseq}
\left( \frac{K_{\lfloor \tau_n t_1 \rfloor}}{\sqrt{\tau_n \log \tau_n}}, \ldots, \frac{K_{\lfloor \tau_n t_N \rfloor}}{\sqrt{\tau_n \log \tau_n}} \right) \Rightarrow
(B_{t_1}, \ldots, B_{t_N}) \; \text{ under } \PP(\cdot | q) \text{ for } \PP\text{-a.e.\ } q.
\end{equation}
\end{prop}

\begin{prop}\label{prop:convas}
Define $i(n) \in \NN$ by $\tau_{i(n)} \le n < \tau_{i(n) + 1}$. 
Then, for any $t > 0$,
\begin{equation}\label{eq:convas}
\frac{K_{\lfloor n t \rfloor}}{\sqrt{n \log n}} - \frac{K_{\lfloor \tau_{i(n)} t \rfloor}}{\sqrt{\tau_{i(n)} \log \tau_{i(n)}}}
\end{equation}
converges in probability to $0$ as $n \to \infty$ under $\mathbb{P}(\cdot | q)$ for $\mathbb{P}$-a.e.\ $q$.
\end{prop}

Propositions~\ref{prop:CLTsubseq} and \ref{prop:convas} are proved in Sections~\ref{sec:proofpropCLTsubseq} and \ref{sec:proofconvas},
respectively. First, we recall in Section~\ref{sec:RW2d} some results about two-dimensional random walks.

\subsection{Two-dimensional random walks}
\label{sec:RW2d}

We gather here some useful facts concerning the local times
of two-dimensional random walks. In the following we always assume (A1) and $d=2$. 
Analogous results hold under the alternative assumption (A1').

\subsubsection{Maximum local times}
\label{subsec:avmaxloctimes}

Let $N_n(x) := \sum_{i=1}^n \mathbf{1}_{\{S_i=x\}}$ be the local times of the random walk $S$ up to time $n$ and
\begin{equation}\label{defmaxloctimes}
N^*_n := \sup_{x \in \ZZ^2} N_n(x)
\end{equation}
be the maximum among them.

\begin{lem}\label{lemma:avmaxloctimes}
\begin{enumerate}
\item[(i)] For all $k \in \NN$, there exists a $K:=K(k) > 0$ such that
\begin{equation}\label{eq:avloctimes}
\EE \left[ (N_n^*)^k \right] \le K (\log n)^{2k} \;\; \forall \; n \geq 2.
\end{equation}
\item[(ii)] There exists a $K > 0$ 
such that
\begin{equation}\label{eq:maxloctimes}
\PP \left(N^*_n > K (\log n)^{2} \right)
\le n^{-2} \;\; \forall \; n \ge 1.
\end{equation}
\end{enumerate}
\end{lem}
\begin{proof}
The two statements follow from Lemma 18(b) in \cite{FMW}.
\end{proof}

\subsubsection{Self-intersection local times}
\label{subsec:SILT}

For $p \in \NN$, the $p$-fold self-intersection local time $I^{[p]}_n$ of $S$ up to time $n$ is defined by
\begin{equation}\label{defI[p]n}
I^{[p]}_n := \sum_{x\in \ZZ^d} N_n^p(x) = \sum_{1 \le i_1, \ldots, i_p \le n} \mathbf{1}_{\{S_{i_1} = \cdots = S_{i_p}\}}.
\end{equation}
When $p=2$ we will omit the superscript and write $I_n$.
\begin{lem}\label{lem:intersecloctimes}
When $d=2$, for all $p \ge 2$ and $k \in \NN$ there exists a $K>0$ such that
\begin{equation}\label{eq:intersecloctimes1}
\EE\left[ (I^{[p]}_n)^k \right] \le K n^k (\log n)^{k(p-1)} \;\; \forall \; n \geq 2.
\end{equation}
\end{lem}

\begin{proof}
The statement can be found in \cite{GuPo13} (Proposition 2.3).
\end{proof}

We will also need the following lemma about the self-intersection local times of higher-dimensional random walks.

\begin{lem}\label{lem:intersecloctimes_d>2}
Let $\widetilde{S}$ be a random walk with a finite, non-singular covariance matrix in dimension $d \ge 3$,
and let $\widetilde{I}_n^{[p]}$ denote its $p$-fold self-intersection local time up to time $n$. 
Then, for all $p \ge 2$ and $k \in \NN$, there exists a $K>0$ such that
\begin{equation}\label{eq:intersecloctimes2}
\EE\left[ (\widetilde{I}^{[p]}_n)^k \right] \le K n^k \;\; \forall \; n \geq 2.
\end{equation}
\end{lem}
\begin{proof}
We can follow the proof of item (i) of Proposition 2.3 in \cite{GuPo13}, 
using the fact that, for all $k \in \NN$, $\sup_n \EE \left[ \widetilde{N}_n(0)^k\right] = \EE \left[ \widetilde{N}_{\infty}(0)^k\right]< \infty$
since $\widetilde{N}_\infty(0) := \sum_{n=1}^{\infty}\mathbf{1}_{\{\widetilde{S}_n=0\}}$ follows a geometric law with parameter $\PP(\widetilde{S}_n\neq 0 \;\forall\; n\ge 1 ) > 0$.
\end{proof}

\subsection{Proof of Proposition \ref{prop:CLTsubseq}}
\label{sec:proofpropCLTsubseq}

\subsubsection{Truncation}
\label{subsubsec:trunc}

Fix $\beta \in (0,1/4)$. For $n \ge 1$, set $b_n := n^\beta$, define $q^{(n)} \in \RR^{\NN^*}$ by
\begin{equation}\label{deftruncscene}
q^{(n)}_i := q_i \mathbf{1}_{\{|q_i| \le b_n \}} -\EE \left[ q_i \mathbf{1}_{\{|q_i| \le b_n \}} \right], \;\;\; i \geq 1,
\end{equation}
and $K^{(n)}$ by
\begin{equation}\label{deftruncpapa}
K^{(n)}_k := \sum_{1\leq i<j\leq k} q^{(n)}_{i} q^{(n)}_j \mathbf{1}_{\{S_i=S_j \}}, \;\;\; k \geq 2.
\end{equation}

The following proposition shows that, in order to prove Proposition~\ref{prop:CLTsubseq} for $K_n$,
it is enough to prove the same statement for $K^{(n)}_n$.

\begin{prop}\emph{(Comparison between $K$ and $K^{(n)}$)}\label{prop:TvsNT}
\text{}

For any $T>0$,
\begin{equation}\label{eq:propTvsNT}	
\lim_{n \to \infty} \sup_{2 \le k \le \lfloor \tau_n T \rfloor} \frac{\left|K_k-K^{(\tau_n)}_k\right|}{\sqrt{\tau_n \log \tau_n}}  = 0 \;\; \PP \text{-a.s.}
\end{equation}
\end{prop}
\begin{proof} 
Let
\begin{equation}\label{eq:propTvsNT1}
q^{(n)>}_{i} := q_i \mathbf{1}_{\{|q_i| > b_n\}} - \EE \left[ q_i \mathbf{1}_{\{|q_i| > b_n\}} \right]
\end{equation}
and note that, since $q_1$ is centered,
\begin{equation}\label{eq:propTvsNT2}
q_i q_j - q^{(n)}_iq^{(n)}_j = - \; q^{(n)>}_{i}q^{(n)>}_{j} + q_i q^{(n)>}_{j} + q^{(n)>}_{i} q_j.
\end{equation}
Write
\begin{equation}\label{eq:propTvsNT3}
K_k - K^{(n)}_k = - \; \mathcal{E}^{(n,1)}_k + \mathcal{E}^{(n,2)}_k + \mathcal{E}^{(n,3)}_k
\end{equation}
where
\begin{equation}\label{eq:propTvsNT4}
\mathcal{E}^{(n,1)}_k := \sum_{1 \le i < j \le k} q^{(n)>}_{i}q^{(n)>}_{j} \mathbf{1}_{\{S_i = S_j\}}
\end{equation}
and $\mathcal{E}^{(n,2)}_k$, $\mathcal{E}^{(n,2)}_k$ are defined analogously from the corresponding terms in \eqref{eq:propTvsNT2}.
Let us focus for the moment on $\mathcal{E}^{(n,1)}_k$. Note that it is a martingale under $\PP$.
Therefore, by Doob's maximal inequality,
\begin{align}\label{eq:propTvsNT5}
\EE \left[ \sup_{2 \le k \le \lfloor n T \rfloor} |\mathcal{E}^{(n,1)}_k|^2 \right] \le \EE \left[ |\mathcal{E}^{(n,1)}_{\lfloor nT \rfloor} |^2 \right]
\le \EE \left[ |q^{(n)>}_1 |^2 \right]^2 \EE \left[ I_{\lfloor n T \rfloor} \right].
\end{align}
Since
\begin{equation}\label{eq:propTvsNT6}
\EE \left[ |q^{(n)>}_1 |^2 \right] \le 2 \EE \left[ |q_1|^2 \mathbf{1}_{\{|q_1| > b_n\}}\right]
\le \frac{C}{b_n^{\gamma-2}},
\end{equation}
by \eqref{eq:propTvsNT4}--\eqref{eq:propTvsNT6} and Lemma~\ref{lem:intersecloctimes}(i) we get
\begin{equation}\label{eq:propTvsNT7}
\EE \left[\sup_{2 \le k \le \lfloor nT \rfloor} \frac{|\mathcal{E}^{(n,1)}_k|^2}{n \log n} \right] \le \frac{C}{n^{2\beta(\gamma-2)}} 
\end{equation}
which is summable along $\tau_n$ since $\gamma > 2$. Analogously, we can show a similar
inequality for $\mathcal{E}^{(n,2)}$ and $\mathcal{E}^{(n,3)}$ with the bound $C n^{-\beta(\gamma-2)}$
instead, which is also summable along $\tau_n$. The proof is concluded by applying the Borel-Cantelli lemma.
\end{proof}

\subsubsection{Decomposition of quenched moments}
\label{subsubsec:decomp_quenchmom}

From now on, we will work with the truncated and recentered version $K^{(n)}$ of the energy.
In fact, for convenience we will work with
\begin{equation}\label{e:notation_quenchedmoms}
2 K_k^{(n)} = \sum_{i \neq \ddot{\imath} \in [k]} q^{(n)}_{i} q^{(n)}_{\ddot{\imath}} \mathbf{1}_{\{ S_{i} = S_{\ddot{\imath}}\}}
\end{equation}
where $[t] := \{1, \ldots, \lfloor t \rfloor\}$.

Fix $\vec{p} = (p_1, \ldots, p_N) \in (\NN^*)^N$, and put $p := |\vec{p}|_1 = p_1 + \ldots + p_N$.
Since the $q^{(n)}_i$ are bounded, the quenched moments
\begin{equation}\label{defquenchedmoments}
m_n^{(\vec{p})} := 
\hspace{-3pt} \sum_{i^1_1 \neq \ddot{\imath}^1_1, \ldots, i^1_{p_1} \neq \ddot{\imath}^1_{p_1} \in [ n t_1 ]} \hspace{-3pt} \cdots \hspace{-3pt} 
\sum_{{i^N_1 \neq \ddot{\imath}^N_1, \ldots, i^N_{p_N} \neq \ddot{\imath}^N_{p_N} \in [n t_N]}} 
\; \prod_{k=1}^N\prod_{\ell=1}^{p_k} q^{(n)}_{i^k_\ell} q^{(n)}_{\ddot{\imath}^k_\ell} 
\; \PP\left( \bigcap_{k=1}^N\bigcap_{\ell=1}^{p_k} \{S_{i^k_\ell} = S_{\ddot{\imath}^k_{\ell}}\}\right)
\end{equation}
are all well defined and satisfy
\begin{equation}\label{m_arequenchedmoments}
m_n^{(\vec{p})} = \mathbb{E}\left[ \prod_{k=1}^N( 2\, K^{(n)}_{\lfloor n t_k \rfloor})^{p_k} \;\middle|\; q \right] \;\; \PP \text{-a.s.}
\end{equation}

We aim to prove that the $m_n^{(\vec{p})}$ when properly normalized converge a.s.\ along $\tau_n$ 
to the corresponding moments of a Gaussian process.
In order to do that, we will first show how they can be decomposed into sums of terms
that are easier to control.

In the following we will use the notation $\mathbf{\tilde{\textbf{\i}}} = (\tilde{\imath}^k_1, \ldots, \tilde{\imath}^k_{p_k})_{k=1}^N$, ${\scriptstyle \sim} \in \{\cdot, \cdot \cdot\}$,
and we will write $\mathbf{i} \neq \mathbf{\ddot{\textbf{\i}}}$ to mean that $\dot{\imath}^k_\ell \neq \ddot{\imath}^k_\ell$ for all $k \in [N]$ and $\ell \in [p_k]$.

Let
\begin{equation}\label{e:defindexset}
\mathcal{I}_{\vec{p}} := \big\{(k,\ell,{\scriptstyle \sim}) \colon\, k \in [N], \ell \in [p_k], {\scriptstyle \sim} \in \{\cdot, \cdot \cdot\} \big\}.
\end{equation}
Note that $|\mathcal{I}_{\vec{p}}| = 2p$.  For $Q \subset \mathcal{I}_{\vec{p}}$, let
\begin{equation}\label{e:defindexQ}
k_Q := \inf \{k \in [N] \colon\, (k,\ell,{\scriptstyle \sim}) \in Q \text{ for some } \ell \in [p_k], {\scriptstyle \sim} \in \{\cdot, \cdot \cdot\} \}
\end{equation}
and, for a collection $\mathcal{Q}$ of subsets of $\mathcal{I}_{\vec{p}}$, let
\begin{equation}\label{e:defspacealldif}
\mathcal{D}_n^{\mathcal{Q}} := \{b = (b_Q)_{Q \in \mathcal{Q}} \colon\, b_Q \in [n t_{k_Q}] \text{ and }b_Q \neq b_P \; \forall \; Q \neq P \in \mathcal{Q} \}.
\end{equation}

For a given pair $\mathbf{i} \neq \mathbf{\ddot{\textbf{\i}}}$, we define a graph structure on $\mathcal{I}_{\vec{p}}$ as follows.
We say that
\begin{equation}\label{e:defgraph1}
(k_1, \ell_1, {\scriptstyle \sim}) \neq (k_2, \ell_2, -) \in \mathcal{I}_{\vec{p}} 
\text{ are adjacent if and only if } \tilde{\imath}^{k_1}_{\ell_1} = \bar{\imath}^{k_2}_{\ell_2}.\\
\end{equation}
Let $\mathcal{P}$ be the resulting partition of $\mathcal{I}_{\vec{p}}$
into connected components according to the graph structure given above.
Then $\mathcal{P}$ belongs to
\begin{equation}\label{e:defspaceofpartitions}
\mathscr{P}_{\vec{p}} := \{\text{partitions } \mathcal{Q} \text{ of } \mathcal{I}_{\vec{p}} \colon\, |\{(k,\ell,\cdot), (k,\ell,\cdot \cdot)\} \cap Q | \le 1 \;\forall\; k \in [N], \ell \in [p_k] \text{ and } Q \in \mathcal{Q}\}.
\end{equation}
Define $a = (a_P)_{P \in \mathcal{P}} \in \mathcal{D}_n^{\mathcal{P}}$ by setting
\begin{equation}\label{e:defa}
a_P = \tilde{\imath}^k_\ell \; \text{ for any } (k,\ell,{\scriptstyle \sim}) \in P.
\end{equation}
Using $t_1 < \cdots < t_N$,
it is straightforward to verify that the map $(\mathbf{i}, \mathbf{\ddot{\textbf{\i}}}) \mapsto (\mathcal{P}, a)$ is a bijection.
Thus we obtain the decomposition
\begin{equation}\label{e:decomp1}
m_n^{(\vec{p})} = \sum_{\mathcal{P} \in \mathscr{P}_{\vec{p}}} \; \sum_{a \in \mathcal{D}_n^{\mathcal{P}}} \; \prod_{P \in \mathcal{P}} (q_{a_P}^{(n)})^{|P|} \,
\PP \left(\bigcap_{k=1}^N\bigcap_{\ell=1}^{p_k} \{S_{i^k_\ell} = S_{\ddot{\imath}^k_\ell}\} \right)
\end{equation}
where in the above $\mathbf{i}, \mathbf{\ddot{\textbf{\i}}}$ are seen as functions of $\mathcal{P}$ and $a$.

Next we define another graph structure, this time on $\mathcal{P}$, as follows.
We say that
\begin{equation}\label{e:defgraph2}
Q \neq P \in \mathcal{P} \text{ are adjacent if and only if } \; \exists\; k \in [N], \ell \in [p_k] \colon\, \{(k, \ell,\cdot), (k, \ell,\cdot \cdot)\} \subset P \cup Q.
\end{equation}
Let $\mathfrak{S}$ denote the partition of $\mathcal{P}$ into connected components according to the graph structure above.
We call $\mathfrak{S}$ the \emph{superpartition} of $\mathbf{i},\mathbf{\ddot{\textbf{\i}}}$.
Note that $|\mathcal{S}| \ge 2$ for all $\mathcal{S} \in \mathfrak{S}$.

We will now show that a consequence of the previous definitions is that
\begin{align}\label{1stconseq_graphstruc}
\bigcap_{k=1}^N\bigcap_{\ell=1}^{p_k} \{S_{i^k_\ell} = S_{\ddot{\imath}^k_{\ell}}\}
& = \bigcap_{\mathcal{S} \in \mathfrak{S}} \bigcap_{P, Q \in \mathcal{S}} \{S_{a_P} = S_{a_Q}\} \nonumber\\
& = \bigcap_{\mathcal{S} \in \mathfrak{S}} \bigcup_{x \in \ZZ^2}\bigcap_{P \in \mathcal{S}} \{S_{a_P} = x\}.
\end{align}

To see this, first note that
\begin{align}
\bigcap_{k=1}^N\bigcap_{\ell =1}^{p_k} \{S_{i^k_\ell}=S_{\ddot{\imath}^k_\ell}\} 
& = \bigcap_{\mathcal{S} \in \mathfrak{S}} \bigcap_{P \in \mathcal{S}} \; \bigcap_{k,\ell \colon (k,\ell, \cdot) \in P} \{S_{a_P} = S_{\ddot{\imath}^k_\ell}\} \nonumber\\
& = \bigcap_{\mathcal{S} \in \mathfrak{S}} \bigcap_{P \in \mathcal{S}} \; \bigcap_{Q \text{ adj.\ } P} \{S_{a_P} = S_{a_Q}\}.
\end{align}
Thus \eqref{1stconseq_graphstruc} will follow once we show that, for all $\mathcal{S} \in \mathfrak{S}$,
\begin{equation}\label{e:claim_conseq}
E_{\mathcal{S}} := \bigcap_{P \in \mathcal{S}} \bigcap_{Q \text{ adj.\ }P}\{S_{a_P}=S_{a_Q}\} \subset \bigcap_{P,Q \in \mathcal{S}} \{S_{a_P} = S_{a_Q}\}. 
\end{equation}
Indeed, for $P,Q \in \mathcal{S}$, there exist $Q_0, \ldots, Q_J \in \mathcal{S}$ such that $Q_0 = P$, $Q_J=Q$ and
$Q_m$ is adjacent to $Q_{m-1}$ for all $m \in [J]$.
We will prove that $E_\mathcal{S} \subset \{S_{a_P} = S_{a_{Q_m}}\}$ for all $m \in [J]$ by induction on $m$.
Since $E_{\mathcal{S}} \subset \{S_{a_{Q_m}} = S_{a_{Q_{m-1}}}\}$ by definition,
the case $m=1$ is covered and, 
supposing that $E_{\mathcal{S}} \subset \{S_{a_{Q_{m-1}}} = S_{a_P}\}$,
we get $E_{\mathcal{S}} \subset \{S_{a_{Q_m}} = S_{a_{Q_{m-1}}}\} \cap \{S_{a_{Q_{m-1}}} = S_{a_P}\} \subset \{S_{a_{Q_{m}}} = S_{a_P}\}$,
proving the induction step. Hence \eqref{e:claim_conseq} is verified, and \eqref{1stconseq_graphstruc} follows.

Thus we see that we may decompose $m_n^{(\vec{p})}$ in the following manner:
\begin{equation}\label{decompmomquenched2}
m_n^{(\vec{p})} = \sum_{\mathcal{P} \in \mathscr{P}_{\vec{p}}} \; m_n^{(\vec{p})}(\mathcal{P})
\end{equation}
where
\begin{equation}\label{defdecomp2}
m_n^{(\vec{p})}(\mathcal{P}) := \sum_{a \in \mathcal{D}_n^{\mathcal{P}}} \; \prod_{P \in \mathcal{P}} (q_{a_P}^{(n)})^{|P|} \,
\PP \left(\bigcap_{\mathcal{S} \in \mathfrak{S}} \bigcap_{P,Q \in \mathcal{S}} \{S_{a_P} = S_{a_Q}\} \right).
\end{equation}

Next, using the identity
\begin{equation}
\prod_{P \in \mathcal{P}} (c_P+ d_P) = \sum_{\mathcal{A} \subset \mathcal{P}} \prod_{P \in \mathcal{A}} c_P \prod_{P \notin \mathcal{A}} d_P,
\end{equation}
we see that we may further decompose $m_n^{(\vec{p})}(\mathcal{P})$ as
\begin{equation}\label{decompmomquenched3}
m_n^{(\vec{p})}(\mathcal{P}) = \sum_{\mathcal{A} \subset\mathcal{P}} m_n^{(\vec{p})}(\mathcal{P},\mathcal{A}),
\end{equation}
where $m_n^{(\vec{p})}(\mathcal{P},\mathcal{A}) :=$
\begin{equation}\label{defdecomp3}
\sum_{a \in \mathcal{D}_n^{\mathcal{P}}} \; \prod_{P \in \mathcal{A}} \mathbb{E} \left[ (q^{(n)}_{a_P})^{|P|} \right]
\prod_{P \notin \mathcal{A}} \left\{ (q^{(n)}_{a_P})^{|P|} - \mathbb{E} \left[ (q^{(n)}_{a_P})^{|P|} \right]\right\} 
\PP \left(\bigcap_{\mathcal{S} \in \mathfrak{S}} \bigcap_{P,Q \in \mathcal{S}} \{S_{a_P} = S_{a_Q}\} \right).
\end{equation}

Let 
\begin{equation}
\mathcal{N} := \{P \in \mathcal{P} \colon\, |P| > 1\}.
\end{equation}
Since $m_n^{(\vec{p})}(\mathcal{P},\mathcal{A}) = 0$ if $\mathcal{A} \cap \mathcal{N}^c \neq \emptyset$
and $m_n^{(\vec{p})}(\mathcal{P},\mathcal{P}) = \EE\left[ m_n^{(\vec{p})}(\mathcal{P}) \right]$,
\begin{equation}\label{decompmomquenched4}
m_n^{(\vec{p})}(\mathcal{P}) - \EE \left[ m_n^{(\vec{p})}(\mathcal{P}) \right] 
= \sum_{\mathcal{A} \subset \mathcal{N} \colon \mathcal{A}^c \neq \emptyset} m_n^{(\vec{p})}(\mathcal{P},\mathcal{A}).
\end{equation}
Moreover, when $\mathcal{A}^c \neq \emptyset$ we may write
\begin{equation}\label{rep_terms}
m_n^{(\vec{p})}(\mathcal{P},\mathcal{A}) = \sum_{ 
(a_P)_{P \notin \mathcal{A}} \in \mathcal{D}^{\mathcal{A}^c}_n
} \; \prod_{P \notin \mathcal{A}} 
\left\{ (q^{(n)}_{a_P})^{|P|} - \EE\left[(q^{(n)}_{a_P})^{|P|}\right] \right\} \mathcal{W}_n\big((a_P)_{P \notin \mathcal{A}}, \mathcal{P}, \mathcal{A}\big),
\end{equation}
where
\begin{equation}\label{defWA}
\mathcal{W}_n\big((a_P)_{P \notin \mathcal{A}}, \mathcal{P}, \mathcal{A}\big) 
:= \sum_{ 
(a_P)_{P \in \mathcal{A}} \in \mathcal{D}^{\mathcal{A}}_n
} \; \prod_{P \in \mathcal{A}} \mathbb{E}\left[(q^{(n)}_{a_P})^{|P|}\right] 
\PP \left(\bigcap_{\mathcal{S} \in \mathfrak{S}} \bigcap_{P,Q \in \mathcal{S}} \{S_{a_P} = S_{a_Q}\} \right) \mathbf{1}_{\{ a \in \mathcal{D}_n^{\mathcal{P}} \}}.
\end{equation}

\subsubsection{Analysis of the terms}
\label{subsubsec:analysis_terms}

We begin with the terms in which $\mathcal{A}^c = \emptyset$, i.e., 
the ones corresponding to $\mathbb{E} [m_n^{(\vec{p})}(\mathcal{P})]$.

\begin{prop}\label{prop:caseAc=0}
For all $\vec{p} \in (\NN^*)^N$, there exists a constant $K \in (0,\infty)$ such that
\begin{equation}\label{eq:caseAc=0}
\left| \mathbb{E} \left[ m_n^{(\vec{p})}(\mathcal{P})\right] \right| \le K (n \log n)^{p/2} \;\; \forall \; n \ge 2,
\end{equation}
where $p := |\vec{p}|_1 = p_1 + \cdots + p_N$.
\end{prop}
\begin{proof}

Integrating \eqref{defdecomp2} we get
\begin{equation}\label{analysis_annealed1}
\begin{aligned}
\EE \left[ m_n^{(\vec{p})}(\mathcal{P}) \right]
& = \prod_{P \in \mathcal{P}} \EE \left [\left(q^{(n)}_1\right)^{|P|} \right] \sum_{a \in \mathcal{D}_n^{\mathcal{P}}} 
\PP \left(\bigcap_{\mathcal{S} \in \mathfrak{S}} \bigcap_{P,Q \in \mathcal{S}} \{S_{a_P} = S_{a_Q}\} \right).
\end{aligned}
\end{equation}
We may suppose that $|P| \ge 2$ for all $P \in \mathcal{P}$ since otherwise
$\EE \left[ m_n^{(\vec{p})}(\mathcal{P}) \right] = 0$. In particular, $|\mathcal{P}| \le p$.
Estimating
\begin{equation}\label{eq:estsceneTR}
\EE\left[|q^{(n)}_1|^{|P|} \right] \le 2^{|P|} \EE\left[|q_1 \mathbf{1}_{\{|q_1|\le b_n\}}|^{|P|} \right]
\le 2^{|P|} \EE\left[|q_1 \mathbf{1}_{\{|q_1|\le b_n\}}|^{|P|-2} q_1^2 \right]
\le 2^{|P|}b_n^{|P|-2},
\end{equation}
we see that the absolute value of the first term with the product in \eqref{analysis_annealed1} is at most $C b_n^{2(p-|\mathcal{P}|)}$.
On the other hand, the second term is smaller than
\begin{align}\label{eq:estim_othercases}
\sum_{a \in [nt_N]^\mathcal{P}} \EE\left[ \prod_{\mathcal{S} \in \mathfrak{S}} \sum_{x \in \ZZ^2} \prod_{P \in \mathcal{S}} \mathbf{1}_{\{S_{a_P} = x\}} \right]  
& = \EE \left[\prod_{ \mathcal{S} \in \mathfrak{S}} \sum_{x \in \ZZ^2} N_{\lfloor n t_N \rfloor}^{|\mathcal{S}|}(x)\right]  \nonumber\\
= \EE \left[\prod_{\mathcal{S} \in \mathfrak{S}} I_{\lfloor n t_N \rfloor}^{[|\mathcal{S}|]}\right] & \le Cn^{|\mathfrak{S}|} (\log n)^{|\mathcal{P}| - |\mathfrak{S}|}
\end{align}
where we used H\"older's inequality and Lemma~\ref{lem:intersecloctimes}(i).

Combining \eqref{analysis_annealed1}--\eqref{eq:estim_othercases} we obtain
\begin{align}\label{estim_annealed}
\left| \EE \left[ m_n^{(\vec{p})}(\mathcal{P}) \right] \right|
 \le C b_n^{2(p-|\mathcal{P}|)} n^{|\mathfrak{S}|} (\log n)^{|\mathcal{P}|-|\mathfrak{S}|}.
\end{align}
We now split into different cases. Note that $|\mathfrak{S}| \le |\mathcal{P}|/2 \le p/2$. 
If $|\mathcal{P}|=p$ and $|\mathfrak{S}| = p/2$, then
\eqref{eq:caseAc=0} holds by \eqref{estim_annealed}.
If $|\mathcal{P}|=p$ and $|\mathfrak{S}| < p/2$, then \eqref{estim_annealed} divided by $(n \log n)^{\frac{p}{2}}$
goes to zero as $n\to\infty$. Lastly, if $|\mathcal{P}|<p$, then
\begin{equation}\label{estim_annealed_last}
(n \log n)^{-\frac{p}{2}} \left| \EE \left[ m_n^{(\vec{p})}(\mathcal{P}) \right] \right|
 \le C n^{-\frac{p-|\mathcal{P}|}{2}\left(1-4\beta\right)} (\log n)^{\frac{p}{2}}
\end{equation}
which goes to zero as $n\to\infty$ since $\beta < 1/4$.
\end{proof}

The rest of the analysis consists in showing that all other terms
with $\mathcal{A}^c \neq \emptyset$ converge to zero a.s.\ along $\tau_n$ when normalized.

\begin{prop}\label{prop:caseAc_notempty}
For any fixed choice of $\vec{p}$ and $\mathcal{P}$, if $\mathcal{A}^c \neq \emptyset$ then
\begin{equation}\label{eq:caseAc_notempty}
\lim_{n \to \infty} \frac{m_{\tau_n}^{(\vec{p})}(\mathcal{P},\mathcal{A})}{(\tau_n \log \tau_n)^{p/2}} = 0 \;\; \mathbb{P} \text{-a.s.}
\end{equation}
\end{prop}

Before we proceed to the proof, we need to introduce a decoupling inequality, due to de la Pe\~na and Montgomery-Smith, 
that will be important for us.

For fixed $\mathcal{A} \neq \mathcal{P}$, 
let $(q^{(n,P)})_{P \notin \mathcal{A}}$ be $|\mathcal{A}^c|$ independent
copies of $q^{(n)}$ and put
\begin{equation}\label{e:defmhat}
\widehat{m}_n^{(\vec{p})}(\mathcal{P}, \mathcal{A}) :=
\sum_{
(a_P)_{P \notin \mathcal{A}} \in \mathcal{D}^{\mathcal{A}^c}_n
} \; \prod_{P \notin \mathcal{A}} 
\left\{ (q^{(n,P)}_{a_P})^{|P|} - \EE\left[(q^{(n)}_{a_P})^{|P|}\right] \right\} \mathcal{W}_n\big((a_P)_{P \notin \mathcal{A}}, \mathcal{P}, \mathcal{A}\big),
\end{equation}
i.e., analogously to \eqref{rep_terms} but with independent copies of $q$ for different $P$.
Then the main theorem in \cite{dlPeMo95} implies that 
there exists a constant $C > 0$ depending on $|\mathcal{A}^c|$ only such that
\begin{equation}\label{e:ineg_decoup}
\PP \left( |m_n^{(\vec{p})}(\mathcal{P}, \mathcal{A})| > u \right)
\le C\, \PP \left( |\widehat{m}_n^{(\vec{p})}(\mathcal{P}, \mathcal{A})| > u/C \right)
\end{equation}
for all $u > 0$.
In particular, we can bound the probability in the l.h.s.\ of \eqref{e:ineg_decoup}
using Markov's inequality and the fact that
\begin{equation}\label{eq:var_est_mA}
\left\|\widehat{m}_n^{(\vec{p})}(\mathcal{P},\mathcal{A})\right\|_2^2 
= \prod_{P \notin \mathcal{A}} \left\| (q^{(n)}_1)^{|P|} - \mathbb{E}\left[(q^{(n)}_1)^{|P|} \right] \right \|_2^2 
\sum_{
(a_P)_{P \notin \mathcal{A}} \in \mathcal{D}^{\mathcal{A}^c}_n
} \mathcal{W}^2_n\big((a_P)_{P \notin \mathcal{A}}, \mathcal{P}, \mathcal{A}\big).
\end{equation}

\begin{proof}[Proof of Proposition~\ref{prop:caseAc_notempty}]
We may suppose that $\mathcal{A} \subset \mathcal{N}$.
Extending the sums in \eqref{defWA} and \eqref{eq:var_est_mA} to $[n t_N]^\mathcal{A}$, respec., $[n t_N]^{\mathcal{A}^c}$,
we may estimate
\begin{equation}\label{estimate_terms_A}
\left\| \widehat{m}_n^{(\vec{p})}(\mathcal{P},\mathcal{A})\right\|_2^2 \le \\
 \prod_{P \notin \mathcal{A}} \left\| (q^{(n)}_1)^{|P|} - \mathbb{E} \left[(q^{(n)}_1)^{|P|} \right]\right\|_2^2 \prod_{P \in \mathcal{A}} \mathbb{E}\left[ (q^{(n)}_1)^{|P|} \right]^2 B_n(\mathcal{P},\mathcal{A})
\end{equation}
where
\begin{align}\label{defBn}
B_n(\mathcal{P},\mathcal{A}) := & \sum_{ 
(a_P)_{P \notin \mathcal{A}} \in [ n t_N]^{\mathcal{A}^c}
} \left\{ \sum_{
(a_P)_{P \in \mathcal{A}} \in [n t_N]^\mathcal{A}
}\EE\left[ \prod_{\mathcal{S} \in \mathfrak{S}} \sum_{x \in \ZZ^2} \prod_{P \in \mathcal{S}} \{S_{a_P} = x\} \right] \right\}^2 \nonumber \\
= & \sum_{
(a_P)_{P \notin \mathcal{A}} \in [ n t_N]^{\mathcal{A}^c}
} \EE\left[ \prod_{\mathcal{S} \in \mathfrak{S}} \sum_{x \in \ZZ^2} N_{\lfloor n t_N \rfloor}^{|\mathcal{S} \cap \mathcal{A}|}(x) \prod_{P \in \mathcal{S}\cap \mathcal{A}^c} 
\mathbf{1}_{\{S_{a_P} = x\}} \right]^2.
\end{align}
We proceed to bound $B_n(\mathcal{P},\mathcal{A})$. 
Denoting by $\widehat{N}_n(x)$ the local times of an independent copy $\widehat{S}$ of $S$, 
and by $\widetilde{N}_n(x,y)$ the local times of the $4$-dimensional random walk $\widetilde{S}_n = (S_n,\widehat{S}_n)$, 
we can rewrite \eqref{defBn} as
\begin{align}\label{est1}
& \sum_{
(a_P)_{P \notin \mathcal{A}} \in [ n t_N]^{\mathcal{A}^c}
} \EE^{\otimes 2}\left[ \prod_{\mathcal{S} \in \mathfrak{S}} 
\sum_{x,y \in \ZZ^2} \left(N_{\lfloor n t_N \rfloor}(x)\widehat{N}_{\lfloor n t_N \rfloor}(y)\right)^{|\mathcal{S}\cap \mathcal{A}|} \prod_{P \in \mathcal{S}\cap \mathcal{A}^c} \mathbf{1}_{\{S_{a_P} = x,\widehat{S}_{a_P}=y\}} \right] \nonumber\\
& = \EE^{\otimes 2} \left[ \prod_{\mathcal{S} \in \mathfrak{S}} \sum_{x,y \in \ZZ^2} \left(N_{\lfloor n t_N \rfloor}(x)\widehat{N}_{\lfloor n t_N \rfloor}(y)\right)^{|\mathcal{S}\cap \mathcal{A}|} \widetilde{N}_{\lfloor n t_N \rfloor}^{|\mathcal{S}\cap \mathcal{A}^c|}(x,y) \right] \nonumber\\
& = \EE^{\otimes 2} \left[ \prod_{\mathcal{S} \subset \mathcal{A}} I_{\lfloor n t_N \rfloor}^{[|\mathcal{S}|]} \widehat{I}_{\lfloor n t_N \rfloor}^{[|\mathcal{S}|]} \prod_{\mathcal{S} \subset \mathcal{A}^c} \widetilde{I}_{\lfloor n t_N \rfloor}^{[|\mathcal{S}|]} \prod_{\stackrel{\mathcal{S} \colon }{\scriptscriptstyle \emptyset \neq \mathcal{S}\cap \mathcal{A} \neq \mathcal{S} }} \sum_{x,y \in \ZZ^2} \left(N_{\lfloor n t_N \rfloor}(x)\widehat{N}_{\lfloor n t_N \rfloor}(y)\right)^{|\mathcal{S}\cap \mathcal{A}|} \widetilde{N}_{\lfloor n t_N \rfloor}^{|\mathcal{S}\cap \mathcal{A}^c|}(x,y) \right] \nonumber\\
& \le \EE^{\otimes 2} \left[ \prod_{\mathcal{S} \subset \mathcal{A}} I_{\lfloor n t_N \rfloor}^{[|\mathcal{S}|]} \widehat{I}_{\lfloor n t_N \rfloor}^{[|\mathcal{S}|]} \prod_{\mathcal{S} \subset \mathcal{A}^c} \widetilde{I}_{\lfloor n t_N \rfloor}^{[|\mathcal{S}|]} \prod_{\mathcal{S} \colon \emptyset \neq \mathcal{S}\cap \mathcal{A} \neq \mathcal{S}} (N_{\lfloor n t_N \rfloor}^* \widehat{N}^*_{\lfloor n t_N \rfloor})^{|\mathcal{S}\cap \mathcal{A}|} \, \widetilde{I}_{\lfloor n t_N \rfloor}^{[|\mathcal{S}\cap \mathcal{A}^c|]} \right].
\end{align}
where $\widehat{I}_n^{[k]}$, $\widetilde{I}_n^{[k]}$ are the analogues of $I_n^{[k]}$ for the corresponding random walks,
and $\widehat{N}_n^* = \sup_x \widehat{N}_n(x)$.
Using H\"older's inequality, Lemmas~\ref{lemma:avmaxloctimes}(i), \ref{lem:intersecloctimes}(i) and \ref{lem:intersecloctimes_d>2}, and the fact that $\mathfrak{S}$ is a partition, we see that \eqref{est1} is at most
\begin{equation}\label{est2}
C n^{|\mathfrak{S}| + |\{\mathcal{S} \colon \mathcal{S}\subset \mathcal{A}\}|} \left(\log n \right)^{2 \left( |\mathcal{A}| + \sum_{\mathcal{S} \colon \emptyset \neq \mathcal{S}\cap \mathcal{A} \neq \mathcal{S}} |\mathcal{S}\cap \mathcal{A}| - |\{\mathcal{S} \colon \mathcal{S} \subset \mathcal{A} \}| \right)} .
\end{equation}

Now we note that $|\mathfrak{S}| \le \lfloor |\mathcal{P}|/2 \rfloor$, $|\{\mathcal{S} \colon \mathcal{S} \subset \mathcal{A} \}| \le \lfloor |\mathcal{A}|/2 \rfloor$ and $2|\mathcal{A}| + |\mathcal{A}^c| \le 2p$.
Therefore,
\begin{equation}\label{deftrou}
t := p - |\mathfrak{S}| - |\{\mathcal{S} \colon \mathcal{S} \subset \mathcal{A}\}| \ge 0
\end{equation}
and there is equality if and only if
\begin{enumerate}\label{conditions_equality}
\item $|\mathcal{A}|$ and $|\mathcal{A}^c|$ are even;
\item $|P| = 2 \; \forall \; P \in \mathcal{A}$ and $|P| = 1 \; \forall \; P \in \mathcal{A}^c $;
\item $|\mathcal{S}|=2 \; \forall \; \mathcal{S} \in \mathfrak{S} $;
\item for any $\mathcal{S} \in \mathfrak{S}$, either $\mathcal{S} \subset \mathcal{A}$ or $\mathcal{S} \subset \mathcal{A}^c$.
\end{enumerate}

Thus by \eqref{defBn}--\eqref{est2}
\begin{equation}\label{summary_estBn}
\frac{B_n(\mathcal{P},\mathcal{A})}{(n \log n)^p} \le \left\{ 
\begin{array}{ll} C(\log n)^{-\frac{|\mathcal{A}^c|}{2}} & \text{ if } t = 0, \\
C n^{-t} (\log n)^{3 p} & \text{ if } t > 0.
\end{array} \right.
\end{equation}

We will consider three cases separately: 
\begin{enumerate}
\item[] Case 1: $t\ge1$;

\item[] Case 2: $t = 0$ and $|\mathcal{A}^c| \ge 4$;

\item[] Case 3: $t=0$ and $|\mathcal{A}^c| = 2$.
\end{enumerate}
For each of these cases we will show that, for every $\epsilon > 0$,
\begin{equation}\label{probas_summable}
\sum_{n=1}^\infty \mathbb{P} \left( |\widehat{m}_{\tau_n}^{(\vec{p})}(\mathcal{P},\mathcal{A})| > \epsilon \sqrt{\tau_n \log \tau_n} \right) < \infty,
\end{equation}
and the result will follow by \eqref{e:ineg_decoup}.
To prove \eqref{probas_summable}, we will use Markov's inequality together with \eqref{estimate_terms_A} and \eqref{summary_estBn}.
In the third case, the variance estimate \eqref{estimate_terms_A} is not good enough, 
but we will get a better bound estimating a higher moment.
\vspace{0.2cm}

\textbf{Case  1:} $(t\ge1)$ For $P \in \mathcal{A}$, we have 
\begin{equation}
\EE\left[ (q^{(n)}_1)^{|P|} \right]^2 \le 2^{2|P|} b_n^{2 (|P| - 2)}
\end{equation}
as in \eqref{eq:estsceneTR} and, for $P \notin \mathcal{A}$, we can estimate in a similar fashion
\begin{align}
\left\| (q^{(n)}_1)^{|P|} - \EE\left[(q^{(n)}_1)^{|P|} \right] \right\|_2^2
& \le 2^{2 |P|} \EE \left[ \left(q_1 \mathbf{1}_{\{|q_1|\le b_n\}}\right)^{2|P|} \right] \nonumber \\
& \le 2^{2 |P|} \mathbb{E} \left[ \left(q_1 \mathbf{1}_{\{|q_1|\le b_n\}}\right)^{2(|P|-1)} q_1^2 \right] \nonumber \\
& \le 2^{2 |P|} b_n^{2 (|P| - 1)}.
\end{align}
Using \eqref{estimate_terms_A}, \eqref{summary_estBn} and 
$|\mathfrak{S}| + |\{\mathcal{S} \colon \mathcal{S} \subset \mathcal{A}\}| \le |\mathcal{A}| + |\mathcal{A}^c|/2$, 
we get
\begin{align}
\frac{\left\|\widehat{m}_n^{(\vec{p})}(\mathcal{P},\mathcal{A})\right\|_2^2}{(n \log n)^p} 
& \le C b_n^{4(p - |\mathcal{A}| -|\mathcal{A}^c|/2)} n^{-t} (\log n)^{3p} \le C \left(\frac{b_n^4}{n}\right)^t (\log n)^{3p} \nonumber\\
& =  C n^{-t\left(1-4\beta \right)} ( \log n)^{3p}
\end{align}
which is summable along $\tau_n$ since $\beta < 1/4$.

\vspace{0.2cm}

\textbf{Case 2:} ($t=0, |\mathcal{A}^c| \ge 4$) 
As mentioned above, in this case $|P|=2$ for $P \in \mathcal{A}$ and $|P| = 1$ for $P \notin \mathcal{A}$.
Using $\mathbb{E}[|q^{(n)}_1|^2] \le \mathbb{E}\left[|q_1|^2\right] = 1$,
we get from \eqref{estimate_terms_A} and \eqref{summary_estBn} that
\begin{equation}\label{eq:estpirestermesTR}
\frac{\left\|\widehat{m}_n^{(\vec{p})}(\mathcal{P}, \mathcal{A})\right\|_2^2}{(n \log n)^p} \le C (\log n)^{-2}
\end{equation}
which is summable along $\tau_n$ since $\alpha > 1/2$. 

\vspace{0.2cm}

\textbf{Case 3:} ($t=0, |\mathcal{A}^c| = 2$)
In this case, \eqref{estimate_terms_A} is not enough to prove \eqref{probas_summable}.
However, since $|\mathcal{A}^c|=2$ and $t=0$, by \eqref{rep_terms} and the discussion after \eqref{deftrou} these terms are of the form
\begin{equation}\label{eq:convpp1}
\widehat{m}_n^{(\vec{p})}(\mathcal{P},\mathcal{A}) 
= \sum_{
a_{P_1} \neq a_{P_2} \in [n t_N]
} q_{a_{P_1}}^{(n,P_1)} q_{a_{P_2}}^{(n,P_2)} \mathcal{W}_n(a_{P_1}, a_{P_2})
\end{equation}
where $P_1, P_2 \in \mathcal{P}$ are such that 
$\mathcal{A}^c = \{P_1, P_2\} \in \mathfrak{S}$ 
and
\begin{equation}\label{def_Wnij}
\mathcal{W}_n(a_{P_1}, a_{P_2}) = \EE\left[ (q^{(n)}_1)^2 \right]^{(|\mathcal{P}|-2)}
\hspace{-10pt} 
\sum_{
(a_P)_{P \in \mathcal{A}} \in \mathcal{D}^{\mathcal{A}}_n
} \PP\left(\bigcap_{\mathcal{S} \in \mathfrak{S}} \bigcap_{P,Q \in \mathcal{S}} \{S_{a_P} = S_{a_Q}\} \right)
\mathbf{1}_{\{a \in \mathcal{D}_n^{\mathcal{P}}\}}.
\end{equation}
Rewrite
\begin{equation}\label{e:Case3rewrite}
\widehat{m}_n^{(\vec{p})}(\mathcal{P},\mathcal{A}) = \sum_{i \neq j \in [n t_N]} q_{i}^{(n)} \widehat{q}_{j}^{(n)} \mathcal{W}_n(i,j)
\end{equation}
where $\widehat{q}$ is an independent copy of $q$.
Since $\mathcal{W}_n(i,j)$ is symmetric and
\begin{equation}\label{convpp_defmart1}
k \mapsto \sum_{i \neq j \in [k]} q_{i}^{(n)} \widehat{q}_{j}^{(n)} \mathcal{W}_n(i,j)
\end{equation}
is a centered martingale, by Burkholder's and Minkowski's inequalities we have
\begin{align}\label{convpp_est1}
\left\|\widehat{m}_n^{(\vec{p})}(\mathcal{P},\mathcal{A})\right\|_{\gamma}^2 
& \le C \sum_{ j = 1 }^{\lfloor n t_N \rfloor} \left\| q_{j}^{(n)} \sum_{i \in [n t_N]\setminus \{ j \}} \widehat{q}_i^{(n)} \mathcal{W}_n(i,j) \right\|_{\gamma}^2 \nonumber\\
& = C \; \EE \left[|q_1^{(n)}|^\gamma \right]^{\frac{2}{\gamma}}\sum_{j=1}^{\lfloor n t_N \rfloor} \left\| \sum_{i \in [n t_N] \setminus \{j\} } q_{i}^{(n)} \mathcal{W}_n(i,j) \right\|_\gamma^2.
\end{align}
By the Marcinkiewicz-Zygmund and Minkowski inequalities, \eqref{convpp_est1} is at most
\begin{align}\label{convpp_est2}
C \sum_{i \neq j \in [nt_N]} \mathcal{W}^2_n(i,j) \le C B_n(\mathcal{P},\mathcal{A}) \le C n^p (\log n)^{p-1},
\end{align}
where we used \eqref{defBn} and \eqref{summary_estBn}.
Combining \eqref{convpp_est1}--\eqref{convpp_est2} we get
\begin{equation}
\frac{\left\|\widehat{m}_n^{(\vec{p})}(\mathcal{P},\mathcal{A})\right\|_{\gamma}^{\gamma}}{(n \log n)^p}
\le C (\log n)^{-\frac{\gamma}{2}}
\end{equation}
which is summable along $\tau_n$ since $\alpha > 2/\gamma$.
\end{proof}

\subsubsection{Conclusion}
\label{subsubsec:concl_proof_propCLTsubseq}

From the results of Section~\ref{subsubsec:analysis_terms} we obtain the following two propositions.
Together with Proposition~\ref{prop:TvsNT}, they will allow us to finish
the proof of Proposition~\ref{prop:CLTsubseq}.

\begin{prop}\label{prop:convannealedmom}\emph{(Convergence of annealed moments)}

For every $\vec{p} \in (\NN^*)^N$,
\begin{equation}\label{eq:convannealedmom}
\lim_{n \to \infty} \EE\left[\prod_{k=1}^N \left( \frac{K_{\lfloor n t_k \rfloor}^{(n)}}{\sqrt{n \log n}} \right)^{p_k} \right]
 = \EE \left[ \prod_{k=1}^N B_{t_k}^{p_k}\right],
\end{equation}
where $B$ is a Brownian motion with variance $\sigma^2$.
\end{prop}
\begin{proof}
First we note that, because of the annealed functional CLT for $K$  (see Appendix \ref{sec:appendixFCLT}) and
Proposition~\ref{prop:TvsNT}, $K^{(n)}$ satisfies a functional CLT with variance $\sigma^2$ under $\PP$. 
Integrating \eqref{decompmomquenched2} and applying Proposition~\ref{prop:caseAc=0},
we see that, for all $\vec{p} \in (\NN^*)^N$,
\begin{equation}\label{proof_convannmom1}
\sup_{n \ge 2} \left| \EE\left[\prod_{k=1}^N \left( \frac{K_{\lfloor n t_k \rfloor}^{(n)}}{\sqrt{n \log n}} \right)^{p_k} \right] \right| 
= 2^{-p} \sup_{n \ge 2} \frac{\left| \EE\left[m_n^{(\vec{p})} \right] \right|}{ (n \log n)^{p/2}} < \infty,
\end{equation}
and hence $\prod_{k=1}^N \left( K_{\lfloor n t_k \rfloor}^{(n)}/\sqrt{n \log n} \right)^{p_k}$ is uniformly integrable for all $\vec{p} \in (\NN^*)^N$.
\end{proof}

\begin{prop}\label{prop:convquenchedmom}\emph{(Convergence of quenched moments)}
\text{}

For every $\vec{p} \in (\NN^*)^N$,
\begin{equation}\label{eq:convquenchedmom}
\lim_{n \to \infty} \EE\left[\prod_{k=1}^N \left( \frac{K_{\lfloor \tau_n t_k \rfloor}^{(\tau_{n})}}{\sqrt{\tau_n \log \tau_n}} \right)^{p_k} \;\middle|\; q \right] 
- \EE\left[\prod_{k=1}^N \left( \frac{K_{\lfloor \tau_n t_k \rfloor}^{(\tau_{n})}}{\sqrt{\tau_n \log \tau_n}} \right)^{p_k}  \right] = 0 \;\; \mathbb{P} \text{-a.s.}
\end{equation}
\end{prop}
\begin{proof}
Combining \eqref{decompmomquenched2} and \eqref{decompmomquenched4},
we see that 
$$ 2^p \left\{ \EE\left[\prod_{k=1}^N \left(K_{\lfloor n t_k \rfloor}^{(n)} \right)^{p_k} \;\middle|\; q \right]
-\EE\left[\left(K_{\lfloor n t_k \rfloor}^{(n)} \right)^{p_k} \right] \right\}= m_{n}^{(\vec{p})}-\EE\left[m_{n}^{(\vec{p})} \right]$$ 
is a sum of terms $m_{n}^{(\vec{p})}(\mathcal{P},\mathcal{A})$ with $\mathcal{A} \subset \mathcal{N}$, $\mathcal{A}^c \neq \emptyset$, so
the result follows from Proposition~\ref{prop:caseAc_notempty}.
\end{proof}

\begin{proof}[Proof of Proposition~\ref{prop:CLTsubseq}]
The conclusion is now straightforward:
Propositions \ref{prop:convannealedmom}--\ref{prop:convquenchedmom} give us
\eqref{eq:CLTsubseq} with $K^{(n)}$ in place of $K$ by the Cram\'er-Wold device and the method of moments, 
and this is passed to $K$ by Proposition~\ref{prop:TvsNT}.
\end{proof}

\subsection{Proof of Proposition \ref{prop:convas}} 
\label{sec:proofconvas}

Before we start, we note some properties of the subsequence $\tau_n $ that will be used in the sequel:
there exist positive constants $K_1$, $K_2$ such that
\begin{equation}\label{propsoussuite}
\begin{array}{ll}
\text{(p1)} & \lim_{n \to \infty}\tau_{n+1}/\tau_n = 1; \\
\text{(p2)} & K_1 \exp\left(n^\alpha /2\right) \le \tau_{n+1} - \tau_n \le K_2 \tau_n/n^{1-\alpha} \;\; \forall \; n \in \NN^*; \\
\text{(p3)} & \tau_n \le K_2 \exp \left(n^\alpha\right) \;\; \forall \; n \in \NN^*.
\end{array}
\end{equation}

\begin{proof}
For integers $b \ge a \ge 2$, let 
\begin{equation}\label{defincrZ}
K_{a,b} := K_b -K_a.
\end{equation}
Once we show that
\begin{equation}\label{eq:aproxsoussuite1}
\lim_{n\to\infty}\sup_{\lfloor \tau_n t \rfloor \le k \le \lfloor \tau_{n+1} t \rfloor}\frac{|K_{\lfloor \tau_n t \rfloor, k }|}{\sqrt{\tau_n \log \tau_n}} = 0 \;\;\; \mathbb{P}\text{-a.s.},
\end{equation}
Proposition~\ref{prop:convas} will follow by noting that
\begin{align}\label{difZseqandsubsubseq}
\left| \frac{K_{\lfloor n t \rfloor}}{\sqrt{n \log n}} - \frac{K_{\lfloor \tau_{i(n)} t \rfloor}}{\sqrt{\tau_{i(n)} \log \tau_{i(n)}}} \right|
& \le \frac{|K_{\lfloor n t \rfloor} - K_{\lfloor \tau_{i(n)} t \rfloor }|}{\sqrt{n \log n}}  \nonumber \\
& \; + \frac{|K_{\lfloor \tau_{i(n)} t \rfloor}|}{\sqrt{\tau_{i(n)} \log \tau_{i(n)}}}\left(1 - \sqrt{\frac{\tau_{i(n)} \log \tau_{i(n)}}{ n \log n}} \right).
\end{align}
Then, by \eqref{eq:aproxsoussuite1} and since $\lim_{n \to \infty}n^{-1}\tau_{i(n)} = 1$, 
the first term in the r.h.s.\ of \eqref{difZseqandsubsubseq} converges a.s.\ to $0$. 
Moreover, the second term converges for $\mathbb{P}$-a.e.\ $q$ in $\mathbb{P}(\cdot | q)$-probability
to $0$ since, by Proposition~\ref{prop:CLTsubseq}, $K_{\lfloor \tau_n t \rfloor}/\sqrt{\tau_n \log \tau_n}$ is a.s.\ tight under $\mathbb{P}(\cdot | q)$.
Therefore, we only need to show \eqref{eq:aproxsoussuite1}. 
By Proposition~\ref{prop:TvsNT}, it is enough to prove \eqref{eq:aproxsoussuite1}
for the sequence $q_{i}^{(\tau_{n+1})}, i\geq 1$, i.e.
\begin{equation}\label{eq:aproxsoussuite1-1}
\lim_{n\to\infty}\sup_{\lfloor \tau_n t \rfloor \le k \le \lfloor \tau_{n+1} t \rfloor }\frac{|K_{\lfloor \tau_n t \rfloor,k}^{(\tau_{n+1})}|}{\sqrt{\tau_n \log \tau_n}} = 0 \;\;\; \mathbb{P}\text{-a.s.}
\end{equation}
To this end, we will make use of a maximal inequality for demimartingales due to Newman and Wright \cite{NeWr82}, as well as Bernstein's inequality.

The sequence $(K_{a,k}^{(n)})_{k \ge a}$ is a zero-mean martingale under both $\PP$ and $\PP(\cdot | S)$ with respect to the filtration
$\sigma ((q_i)_{i\le k}, S )$. 
Indeed, 
$$K_{a, k+1}^{(n)} -K_{a ,k}^{(n)}= K_{k+1}^{(n)}- K_{k}^{(n)} = 
q_{k+1}^{(n)} \sum_{i=1}^k q_i^{(n)} 1_{\{S_i=S_{k+1}\}}$$
and the r.v.'s $q_i^{(n)}$, $i\geq 1$ are independent and centered.
Therefore, 
\begin{align}\label{est_cond_variance}
\EE \left[ \left(K^{(n)}_{a,b}\right)^2 \;\middle|\; S\right] 
& = \sum_{k=a+1}^b \EE \left[ \left(q_{k}^{(n)} \sum_{i=1}^{k-1} q_i^{(n)} 1_{\{S_i=S_{k}\}} \right)^2 \;\middle|\; S \right] \nonumber\\
& \le C \sum_{k=a+1}^b \sum_{i=1}^{k-1} \mathbf{1}_{\{S_i = S_k\}} = C \sum_{k=a+1}^b N_{k-1}(S_k) \nonumber\\
& \le C \sum_{x \in \ZZ^2} N_b(x) N_{a,b}(x) \le C \sqrt{I_{b}} \sqrt{I_{a,b}}
\end{align}
where $N_{a,b}(x) := \sum_{k=a+1}^b \mathbf{1}_{\{S_k = x\}}$, $I_{a,b} := \sum_{x \in \ZZ^2} N^2_{a,b}(x)$
and for the last step we used the Cauchy-Schwarz inequality.
Integrating \eqref{est_cond_variance} and using H\"older's inequality we get
\begin{align}\label{est_variance}
\EE \left[ \left(K^{(n)}_{a,b}\right)^2\right] 
& \le C \EE \left[\sqrt{I_{b}} \sqrt{I_{a,b}} \right] \le C \sqrt{\EE \left[ I_{b} \right] \EE \left[ I_{a,b} \right]} \nonumber\\
& \le C \sqrt{b \log b} \sqrt{(b-a) \log (b-a)} \le C b \log b,
\end{align}
where for the third inequality we used Lemma~\ref{lem:intersecloctimes}(i) and that $I_{a,b}$ has the same law as $I_{b-a}$.

Since $K^{(n)}_{a,k}$ is in particular a demimartingale under $\PP$, by Corollary 6 in \cite{NeWr82} we get
\begin{align}\label{max_ineq}
& \PP\left( \sup_{\lfloor \tau_n t \rfloor \leq k \leq \lfloor  \tau_{n+1} t \rfloor} \vert K^{(\tau_{n+1})}_{\lfloor \tau_n t \rfloor ,k} \vert \geq  2 \varepsilon \sqrt{\tau_n \log \tau_n}\right) \nonumber \\
& \qquad \qquad \qquad \leq \sqrt{\frac{2 \EE\left[\left(K^{(\tau_{n+1})}_{\lfloor \tau_n t \rfloor, \lfloor \tau_{n+1} t \rfloor}\right)^2 \right]}{ \varepsilon^2 \tau_n \log \tau_n}}
\sqrt{\PP \left( |K^{(\tau_{n+1})}_{\lfloor \tau_n t \rfloor, \lfloor \tau_{n+1} t \rfloor}| \geq \varepsilon \sqrt{\tau_n \log \tau_n} \right)} \nonumber\\
& \qquad \qquad \qquad \le C \sqrt{\PP \left( |K^{(\tau_{n+1})}_{{\lfloor \tau_n t \rfloor, \lfloor \tau_{n+1} t \rfloor}}| \geq \varepsilon \sqrt{\tau_n \log \tau_n} \right)}
\end{align}
by \eqref{est_variance} and the properties of $\tau_n$.

Now note that $K_{a,b}^{(n)}$ can be rewritten as 
$$ \sum_{x\in \ZZ^2} \Big(\Lambda_{b} (x) -\Lambda_{a}(x)\Big) $$
where $$\Lambda_k(x) = \sum_{i < j \in \mathcal{L}_k(x) } q_i^{(n)} q_j^{(n)}, \quad \mathcal{L}_k(x) := \{1 \le i \le k \colon S_i = x\}. $$
Given the random walk $S$, the random variables $\Lambda_{b}(x) -\Lambda_{a}(x), x\in \ZZ^2$ are independent, centered and 
uniformly bounded by $(b_{n} N_b^*)^2.$ 
Furthermore, by \eqref{est_cond_variance},
\begin{equation}\label{varcond}
\sum_{x\in \ZZ^2} \EE[ (\Lambda_{b} (x) -\Lambda_{a} (x))^2|S] \leq   C(\sqrt{I_{a,b}} \sqrt{I_{b}} ) .
\end{equation}
Thus we may use Bernstein's inequality under $\PP(\cdot |S)$ to estimate the probability in the right hand side of \eqref{max_ineq}, obtaining that, for all $u>0$,
$$\PP \left( | K_{a,b}^{(n)} |  \geq u \;\middle|\; S \right) 
\leq \exp \left(  -\frac 1 2 \frac{ u^2}{\sqrt{I_{a,b} I_{b}}+ u (b_{n} N_b^*)^2} \right).$$
Integrating with respect to the random walk, we get, for every $\varepsilon > 0$,
\begin{align}\label{exp}
& \PP \left( | K_{{\lfloor \tau_n t \rfloor, \lfloor \tau_{n+1} t \rfloor}}^{(\tau_{n+1})} |  \geq  \varepsilon \sqrt{\tau_n \log \tau_n} \right) \nonumber\\
& \qquad \leq  \EE \left[\exp\left(  -C \frac{  \tau_n \log \tau_n }{ \sqrt{I_{{\lfloor \tau_n t \rfloor, \lfloor \tau_{n+1} t \rfloor}} I_{\lfloor \tau_n t \rfloor}} + \sqrt{\tau_n \log \tau_n} (b_{\tau_{n+1}} N^*_{\lfloor \tau_{n+1} t \rfloor})^2 }\right)\right].
\end{align}

Recall that, by Lemma~\ref{lemma:avmaxloctimes}(ii), there exists $C>0$ such that
\begin{equation}\label{eq:estim_prob_N*}
P \left(N^*_k > C (\log k)^{2} \right)
\le k^{-2} \;\; \forall \; k \ge 1.
\end{equation}

Now fix $0 < \delta < \frac{1}{2}(\alpha^{-1} - 1)$ and an integer $\theta > 2/(\alpha \delta)$.
By Markov's inequality and Lemma~\ref{lem:intersecloctimes}(i), we have
\begin{equation}\label{eq:estim_prob_In}
P \left( I_k > k (\log k)^{1+\delta} \right) 
\le \frac{E \left[ I_k^{\theta}\right]}{ k^{\theta} (\log k)^{(1+\delta)\theta}}
\le \frac{C}{(\log k)^{\theta \delta}} \;\; \forall \; k \ge 2.
\end{equation}

By \eqref{exp}--\eqref{eq:estim_prob_In}, the subadditivity of $\sqrt{\cdot}$ and the fact that $e^{-2x/(y+z)} \le e^{-x/y}+ e^{-x/z}$ for any $x,y,z>0$, we see that \eqref{max_ineq} is at most
\begin{equation}\label{eq:aproxsoussuite_ineg2}
C_1(\tau_{n+1})^{-1}+C_2(\log (\tau_{n+1} - \tau_{n}))^{-\frac{\theta \delta}{2}} +C_3 (\log \tau_n)^{-\frac{\theta \delta}{2}} 
+ e^{ -C_4 d_n/e_n}+ e^{-C_4 d_n/f_n},
\end{equation}
where $C_1$--$C_4$ are positive constants and
\begin{align}
d_n & := \tau_n \log \tau_n, \nonumber\\
e_n & := \sqrt{\tau_n(\tau_{n+1}-\tau_n)} [\log(\tau_{n+1}-\tau_{n}) \log(\tau_n)]^{(1+\delta)/2}, \nonumber\\ 
f_n & := \sqrt{\tau_n \log \tau_n} (\tau_{n+1})^{2\beta} (\log \tau_{n+1})^4.
\end{align}
Using the properties of $\tau_n$, we see that the first term of \eqref{eq:aproxsoussuite_ineg2} is summable;
by our choice of $\theta$, so are the second and the third. Furthermore,
\begin{equation}
d_n/e_n \ge C n^{\frac{1-\alpha}{2}}/(\log\tau_n)^{\delta} \ge C n^{\frac{1-\alpha(1+2\delta)}{2}},
\end{equation}
and so the fourth term is summable by our choice of $\delta$. As for the last term,
note that
\begin{equation}
d_n/f_n \ge C (\tau_n)^{\frac12(1 - 4\beta)} (\log {\tau_n})^{-7/2}
\end{equation}
so the fifth term is summable since $\beta < 1/4$.
Thus, by the Borel-Cantelli lemma, \eqref{eq:aproxsoussuite1} holds.
\end{proof}

\section{Case of dimensions three and higher}
\label{sec:dim>2}

\subsection{Assumptions and results}
\label{subsec:assumpresultsd>2}

In $d\ge3$ we can relax the condition $\gamma > 2$ used in Section~\ref{sec:dim2}.
Here we will only assume (A1) and
\begin{equation}\label{assumptionsd>2}
\EE \left[ q_1\right] =0, \quad \EE \left[ q_1^2\right] =1.
\end{equation}
However, we will need to recenter $K_n$, since its quenched expectation
is not subdiffusive as can be checked with a simple computation.

\begin{theorem}\label{theoCLT3d}
In $d\ge3$, under assumptions (A1) and \eqref{assumptionsd>2}, for $\mathbb{P}$-a.e.\ $q$,
the process
\begin{equation}\label{eq:theoCLT3d}
\frac{K_{\lfloor n t \rfloor} - \EE\left[ K_{\lfloor n t \rfloor} | q \right]}{\sqrt{n}}, \;\;\; t \ge 0, 
\end{equation}
converges under $\PP(\cdot | q)$ in the Skorohod topology 
to a Brownian motion with variance
\begin{equation}\label{var_3d}
\sigma^2 = \sum_{i=1}^{\infty} \PP(S_i=0)\PP(S_i \neq 0).
\end{equation}
\end{theorem}

\noindent\textbf{Remark:}
The result above remains true for any transient random walk on $\ZZ^d$, $d\ge1$, as long as
$$\sum_{n=1}^\infty \rho^{-n} \sum_{k=1}^{\rho^n} \sum_{i=k}^\infty \PP(S_i=0) < \infty $$
for some $\rho > 1$.

\begin{proof}
The idea is to approximate $K_n-\EE[K_n|q]$ by an additive functional of a Markov chain
and then apply known results in this setting.
Let
\begin{equation}\label{def:sequence_pairs}
(X(k), q(k)), \;\; k \in \ZZ,
\end{equation}
be an i.i.d.\ sequence with each term distributed as $(S_1,q_1)$, and denote its law by $\PP$ and its expectation by $\EE$.
For a time $l \in \NN$, define the sequences $\mathcal{X}_l$ and $q_l$ by
\begin{equation}\label{defseq_time}
\begin{array}{lcll}
q_l(k) & := & q(l+k), & k \in \ZZ,\\
\mathcal{X}_l(k) & := & X(k+l), & k \le 0,
\end{array}
\end{equation}
and put
\begin{equation}\label{defMkvchain}
\xi_l := (q_l, \mathcal{X}_l).
\end{equation}
Then $\xi_l$ is a Markov chain on the state space $\RR^{\ZZ} \times (\RR^d)^{\ZZ_-}$.
Moreover, the process $\xi_l$ is stationary and ergodic under $\PP$.

For $\mathcal{X} \in (\RR^d)^{\ZZ_-}$ and $i \le k \le 0$, define
\begin{equation}\label{defsigmaop}
\Sigma_i^k(\mathcal{X}) := \mathcal{X}(i+1) + \cdots + \mathcal{X}(k).
\end{equation}
Then, writing
\begin{equation}
K_n = \sum_{k=2}^n q(k) \sum_{i=1}^{k-1} q(k-i) \mathbf{1}_{\{\Sigma_{-i}^0(\mathcal{X}_k)=0\}},
\end{equation}
we see that
\begin{equation}
K_n - \EE [K_n | q]= \sum_{l=2}^n H_l(\xi_l)
\end{equation}
where
\begin{equation}\label{defHt}
H_l(q,\mathcal{X}) := q(0) \sum_{i=1}^{l -1} q(-i) \left\{ \mathbf{1}_{\{\Sigma_{-i}^0(\mathcal{X})=0\}} - \PP (S_i = 0) \right\}.
\end{equation}
Since $d\ge3$, $\sum_{i=1}^{\infty}\PP(S_i=0)<\infty$ and we may define
\begin{equation}\label{defH}
H(q,\mathcal{X}) := q(0) \sum_{i=1}^{\infty} q(-i) \left\{ \mathbf{1}_{\{\Sigma_{-i}^0(\mathcal{X})=0\}} - \PP (S_i = 0) \right\}.
\end{equation}
\begin{lem}\label{lem:approx}
For each $T>0$,
\begin{equation}\label{eq:approx}
\lim_{n \to \infty }\frac{1}{\sqrt{n}} \sup_{1 \le k \le nT} \left|\sum_{l=2}^k(H_l(\xi_l) - H(\xi_l))\right| = 0 \;\;\; \PP\text{-a.s.}
\end{equation}
\end{lem}
\begin{proof}
To start, note that
\begin{equation}
\sum_{l=2}^k(H(\xi_l) - H_l(\xi_l))
= \sum_{l=2}^k q(l) \sum_{i=l}^{\infty} q(l-i) \left(\mathbf{1}_{\{\Sigma_{-i}^0(\mathcal{X}_l)=0\}} - \mathbb{P}(S_i = 0) \right)
\end{equation}
is a martingale with respect to the filtration $(\mathcal{F}_k)_{k\geq 1}$ where $\mathcal{F}_k = \sigma( (q(l),X(l)) , l \leq k)$ under $\PP$. By Doob's maximal inequality,
\begin{align}\label{eq:approx2}
\EE \left[\frac{1}{n}\sup_{1 \le k \le n} \left|\sum_{l=2}^k (H_l(\xi_l) - H(\xi_l))\right|^2 \right]
& \le \frac{1}{n} \sum_{l=2}^n \sum_{i=l}^{\infty} \text{Var}\left(\mathbf{1}_{\{\Sigma_{-i}^0(\mathcal{X}_l)=0\}} \right) \nonumber\\
& \le \frac{1}{n} \sum_{l=2}^n \sum_{i=l}^{\infty} \PP(S_i=0) \nonumber\\
& \le C n^{-1/2}.
\end{align}
The last step follows from the bound $\PP(S_n = 0) \le C n^{-d/2}$  (see e.g.\ \cite{JaPr}, Lemma 1). 
Therefore, by the Borel-Cantelli lemma, \eqref{eq:approx} holds with the sequence $2^n$ in place of $nT$.
The result is passed to the original sequence by considering, for each $n$, $k_n$ such that $2^{k_n-1} \le nT < 2^{k_n}$.
\end{proof}

Because of Lemma~\ref{lem:approx}, the theorem will follow
once we show the same statement for the additive functional
\begin{equation}
\mathbb{H}_n := \sum_{l=2}^n H(\xi_l).
\end{equation}

\begin{lem}\label{lem:subdiff}
\begin{equation}
\sup_{n \ge 2} \mathbb{E} \left[ \mathbb{E} \left[\mathbb{H}_n \; \middle| \xi_0 \; \right]^2 \right] < \infty.
\end{equation}
\end{lem}

\begin{proof}
We have
\begin{equation}
\mathbb{E} \left[ H(\xi_l) \;\middle|\; \xi_0 \right]
= q_l(0) \sum_{i=1}^{\infty}q_l(-i) \Big\{ \mathbb{P} \left( \Sigma_{-i}^0(\mathcal{X}_l)=0\;\middle|\; (X_j)_{j \le 0}\right) - \mathbb{P} \left(S_i = 0 \right) \Big\}.
\end{equation}
Since $\mathbb{P} \left( \Sigma_{-i}^0(\mathcal{X}_l)=0\;\middle|\; (X_j)_{j \le 0}\right) = \mathbb{P} \left(S_i = 0 \right)$
if $i \le l$, we have
\begin{align}\label{estim0}
\mathbb{E} \left[ \mathbb{E} \left[\mathbb{H}_n \; \middle| \xi_0 \; \right]^2 \right]
& = \sum_{l=2}^n \sum_{i=l+1}^{\infty} \text{Var}\left( \mathbb{P} \left( \Sigma_{-i}^0(\mathcal{X}_l)=0\;\middle|\; (X_j)_{j \le 0}\right) \right) \nonumber\\
& \le \sum_{l=2}^n \sum_{i=l+1}^{\infty} \mathbb{E}\left[ \mathbb{P} \left( \Sigma_{-i}^0(\mathcal{X}_l)=0\;\middle|\; (X_j)_{j \le 0}\right)^2 \right] \nonumber\\
& = \sum_{l=2}^n \sum_{k = 0}^{\infty} \mathbb{P}^{\otimes 3} \left( S^{(1)}_l = S^{(2)}_l = - S^{(3)}_{k+1}\right)
\end{align}
where $S^{(j)}$, $j=1,2,3$ are independent copies of $S$ with joint law $\PP^{\otimes 3}$.
This last line is equal to
\begin{equation}
\sum_{l=2}^n \sum_{x \in \ZZ^d} \mathbb{P}^{\otimes 2} \left( S^{(1)}_l = S^{(2)}_l = x \right) \mathbb{E} \left[ N_{\infty}(-x) \right]
\le \mathbb{E} \left[ N_{\infty}(0) \right] \sum_{l=1}^\infty \mathbb{P}^{\otimes 2} \left( S^{(1)}_l = S^{(2)}_l \right) < \infty
\end{equation}
since  the last sum is the total local time at $0$ of the $d$-dimensional random walk $S^{(1)}_l - S^{(2)}_l$.
\end{proof}

Now the theorem readily follows from Lemmas~\ref{lem:approx}--\ref{lem:subdiff} together with e.g.\ the main theorem in \cite{DeLi}
(note that $\EE[\mathbb{H}_n | \xi_0] = \sum_{k=2}^n P^k H(\xi_0)$ with their notations),
the fact that $(q(k))_{k \ge 1}$ is measurable with respect to $\sigma(\xi_0)$,
and a straightforward calculation of the variance of $\mathbb{H}_n$.
\end{proof}

\appendix
 
\section{Functional CLT under the conditional law given $S$} 
\label{sec:appendixFCLT}
 
In this appendix we prove Theorem~\ref{thm:FCLTgivenS}.

\begin{proof}
We will apply the martingale functional CLT in the Lindeberg-Feller formulation
as in e.g.\ \cite{Du05}, Theorem 7.3 on page 411. We will tacitly use the laws of large numbers
for $I^{[p]}_n$, $p \ge 0$, proven in \cite{Cer07} for $d=2$ and \cite{BeKo09} for $d\ge3$.
Note that the result in \cite{Cer07} is also valid under (A1').

Let us define $K_1 := 0$ and
\begin{align}\label{eq:FCLT1}
\Delta_{n,k} := s_n^{-1}(K_k - K_{k-1}) = s_n^{-1} q_k \; \sum_{i=1}^{k-1} q_i \mathbf{1}_{\{S_i = S_k\}}, \;\; k \ge 2.
\end{align}
Then $\Delta_{n,m}$ is a martingale difference array under $\mathbb{P}(\cdot | S)$ w.r.t.\ the filtration
$\mathcal{F}_{m} := \sigma(q_i, i \le m)$. The corresponding quadratic variations are given by
\begin{equation}
Q_{n,m} :=  \sum_{k=1}^m \EE \left[ \Delta_{n,k}^2 \;\middle|\; S, \mathcal{F}_{k-1} \right].
\end{equation}
According to \cite{Du05},
the proof will be finished once we show that, for all $\epsilon > 0$,
\begin{equation}\label{lind}
\lim_{n \to \infty} \sum_{k=2}^n \mathbb{E} \left[ \Delta_{n,k}^2 \mathbf{1}_{\{|\Delta_{n,k}|> \epsilon\}}\;\middle|\; S \right] = 0
\;\;\; \PP\text{-a.s.}
\end{equation}
and that, for all $t \ge 0$,
\begin{equation}\label{convvar}
\lim_{n \to \infty} Q_{n, \lfloor nt \rfloor} = \sigma^2 t \;\; \text{ in probability under } \PP(\cdot | S) \text{ for } \PP\text{-a.e.\ } S.
\end{equation}
In fact, the theorem in \cite{Du05} concludes convergence for $t \in [0,1]$,
but it this then easy to extend the result to $t \in [0, T]$ with $T \in \NN$ and thus to $t \in [0,\infty)$.

We proceed to verify \eqref{lind}--\eqref{convvar}, starting with the first. 
Write
\begin{align}
\sum_{k=2}^n \mathbb{E} \left[ \Delta_{n,k}^2 \mathbf{1}_{\{|\Delta_{n,k}|> \epsilon \}}\;\middle|\; S \right]
& \le C \sum_{k=2}^n \mathbb{E} \left[ |\Delta_{n,k}|^\gamma \;\middle|\; S \right] \nonumber\\
& = C s_n^{-\gamma} \sum_{k=2}^n \mathbb{E} \left[ \left|\sum_{i=1}^{k-1} q_i \mathbf{1}_{\{S_i = S_{k}\}} \right|^\gamma\;\middle|\; S \right] \nonumber\\
& \le C s_n^{-\gamma} \sum_{k=2}^n \left(\sum_{i=1}^{k-1} \mathbf{1}_{\{S_i = S_{k}\}} \right)^{\gamma/2} \nonumber\\
& = C s_n^{-\gamma} \sum_{x \in \ZZ^d} \sum_{k=2}^n \mathbf{1}_{\{S_{k}=x\}} N_{k-1}^{\gamma/2}(x) \nonumber\\
& \le C s_n^{-\gamma} I_n^{[1+\gamma/2]},
\end{align}
where for the third step we used the Marcienkiewicz-Zygmund and Minkowski inequalities. 
Since $I_n^{[1+\gamma/2]}/s_n^\gamma$ goes to $0$ a.s.\ as $n \to \infty$, \eqref{lind} follows.

Let us now verify \eqref{convvar}. Write
\begin{equation}\label{eq:decompQ}
Q_{n,m} = s_n^{-2} \left( \frac{I_m - m}{2} + R^{(1)}_m + R^{(2)}_m \right),
\end{equation}
where
\begin{equation}
\begin{aligned}
R^{(1)}_m := & \sum_{k=2}^{m} \sum_{i=1}^{k-1}(q_i^2-1) \mathbf{1}_{\{S_i = S_k\}} \\
= & \sum_{i=1}^{m-1} (q_i^2-1) N_{i,m}(S_i)
\end{aligned}
\end{equation}
with $N_{a,b}(x) := \sum_{k=a+1}^b \mathbf{1}_{\{S_k=x\}}$, and
\begin{equation}
\begin{aligned}\label{eq:FCLT2}
R^{(2)}_m := & \; 2 \sum_{k=2}^{m} \sum_{1 \le i < j \le k-1} q_i q_j \mathbf{1}_{\{S_i = S_j = S_k\}}\\
= & \; 2 \sum_{x \in \ZZ^d}\sum_{1 \le i < j \le m-1} q_i q_j \mathbf{1}_{\{S_i=S_j = x\}} N_{j,m}(x).
\end{aligned}
\end{equation}
Since $(I_{\lfloor nt \rfloor}-\lfloor nt \rfloor)/2 s_n^2 \to \sigma^2 t$ a.s.,
we only need to show that the remaining terms in \eqref{eq:decompQ} converge to $0$.
Note that in dimension $d\geq 3$, $\sigma^2$ can be written as
$$\sigma^2= \frac{1}{2}\left(\sum_{j=1}^{\infty}j^2 \chi^2 (1-\chi)^{j-1}-1\right)$$
where $\chi := \mathbb{P} \left( S_n \neq 0 \; \forall \; n \ge 1\right)$.

Let us first deal with $R^{(2)}_m$. Note that, under $\mathbb{P}(\cdot |S)$, the summands in the r.h.s.\ of \eqref{eq:FCLT2} are independent and centered for different $x$ to write
\begin{align}
\mathbb{E} \left[ (R^{(2)}_m)^2\;\middle|\; S \right]
& = 4 \sum_{x \in \ZZ^d} \sum_{1 \le i < j \le m-1} \mathbf{1}_{\{S_i=S_j = x\}} N^2_{j,m}(x) \nonumber\\\
& \le C \sum_{x \in \ZZ^d} N^{4}_m(x) = C I_m^{[4]}
\end{align}
and conclude that $R^{(2)}_{\lfloor nt \rfloor} /s_n^2$ goes to $0$ in probability under $\mathbb{P}(\cdot | S)$.
To control $R^{(1)}_m$, we split into two cases. If $\gamma\geq 4$, then reasoning as before we get 
$$\mathbb{E} \left[ (R^{(1)}_m)^2\;\middle|\; S \right]
 \le  C I_m^{[3]}$$
 and we conclude as for $R^{(2)}_m$.
If $\gamma<4$, we use Theorem 1(c) in \cite{HaWr71}.
Note that
\begin{align}
\sum_{i=1}^{m-1} N^{\gamma/2}_{i,m}(S_i)
& = \sum_{x \in \ZZ^d} \sum_{i=1}^{m-1} \mathbf{1}_{\{S_i = x \}}N^{\gamma/2}_{i,m}(x) \nonumber\\
& \le \sum_{x \in \ZZ^d} N^{1+\gamma/2}_m(x) = I^{[1+\gamma/2]}_m,
\end{align}
and also 
\begin{equation}
\mathbb{P}\left( |q^2-1| \ge u \right) \le C u^{-\gamma/2} \;\;\; \forall \; u > 0.
\end{equation}
Letting 
\begin{equation}
a_{n,i} := \left\{ \begin{array}{ll}
N_{i,\lfloor nt \rfloor}(S_i)/s_n^2 & \text{ if } i \le \lfloor nt \rfloor,\\
0 & \text{ otherwise,}
\end{array}\right.
\end{equation}
we obtain from the aforementioned theorem that, for some constant $C>0$,
\begin{equation}
\mathbb{P} \left( |R^{(1)}_{\lfloor nt \rfloor}| > \epsilon s_n^2  \;\middle|\; S \right) \le C I^{[1 + \gamma/2]}_{\lfloor nt \rfloor}/s_n^{\gamma},
\end{equation}
which goes to $0$ as $n \to \infty$. 
\end{proof}

\end{document}